\documentclass{siamart116}

\usepackage[utf8]{inputenc}
\usepackage{amsmath,amssymb,enumerate,framed, enumitem}
\usepackage{color,colortab,epsfig}
\usepackage{tikz, float}
\RequirePackage{fix-cm}

\newcommand{\N}{\mathbb{N}}
\newcommand{\Z}{\mathbb{Z}}
\newcommand{\R}{\mathbb{R}}
\newcommand{\conv}{\operatorname{conv}}

\newcommand{\setcond}[2]{\left\{ #1 \,:\, #2 \right\}}

\newcommand{\intr}{\operatorname{int}}
\newcommand{\relintr}{\operatorname{relint}}

\newcommand{\pyr}{\operatorname{Pyr}}
\newcommand{\proj}{\operatorname{proj}}
\newcommand{\col}{\operatorname{col}}

\newcommand{\cone}{\operatorname{cone}}

\newcommand{\cK}{\mathcal{K}}
\newcommand{\cR}{\mathcal{R}}
\newcommand{\cP}{\mathcal{P}}
\newcommand{\cL}{\mathcal{L}}

\newcommand{\cF}{\mathcal{F}}
\newcommand{\cX}{\mathcal{X}}

\DeclareMathOperator*{\argmax}{arg\,max}

\renewcommand{\epsilon}{\varepsilon}

\def\ve#1{\mathchoice{\mbox{\boldmath$\displaystyle\bf#1$}}
{\mbox{\boldmath$\textstyle\bf#1$}}
{\mbox{\boldmath$\scriptstyle\bf#1$}}
{\mbox{\boldmath$\scriptscriptstyle\bf#1$}}}

\newcommand{\0}{{0}}

\newcommand{\p}{{\ve p}}

\newtheorem{CLAIM}[theorem]{Claim}
\newenvironment{cpf}{\begin{trivlist} \item[] {\em Proof of Claim.}}{\hspace*{\stretch{1}} $\diamond$ \end{trivlist}}

\numberwithin{equation}{section}

\ifpdf
  \DeclareGraphicsExtensions{.eps,.pdf,.png,.jpg}
\else
  \DeclareGraphicsExtensions{.eps}
\fi

\newcommand{\TheTitle}{Non-unique lifting of integer variables in minimal inequalities} 
\newcommand{\TheAuthors}{Amitabh Basu and Santanu S. Dey and Joseph Paat}

\headers{\TheTitle}{\TheAuthors}

\title{{\TheTitle}\thanks{Submitted to the editors 16 February 2017.
\funding{Amitabh Basu and Joseph Paat were supported in part by the NSF grant CMMI1452820. Santanu S. Dey gratefully acknowledges the support from NSF grant CMMI 1149400.}}}

\author{
  Amitabh Basu \thanks{Dept. of Applied Mathematics and Statistics, The Johns Hopkins University, (\email{basu.amitabh@jhu.edu}).}
  \and
  Santanu S. Dey\thanks{School of Industrial and Systems Engineering, Georgia Institute of Technology, (\email{santanu.dey@isye.gatech.edu}).}
  \and
  Joseph Paat\thanks{Institute for Operations Research, ETH Z\"urich, (\email{joseph.paat@ifor.math.ethz.ch}).}
  }

\usepackage{amsopn}

\begin{document}

\maketitle

\begin{abstract}
We explore the lifting question in the context of cut-generating functions. 
Most of the prior literature on this question focuses on cut-generating functions that have the unique lifting property. 
We develop a general theory for understanding the lifting question for cut-generating functions that do not necessarily have the unique lifting property.
\end{abstract}

\begin{keywords}
  cutting plane theory, cut-generating functions, lattice-free sets
\end{keywords}

\begin{AMS}
90C10, 90C11
\end{AMS}

\section{Introduction}\label{sec:intro}
Let $S\subseteq \R^n \setminus\{0\}$ be a closed set and consider the model
\begin{equation}\label{def mixed-int set}
	X_{S}(R,P) := \setcond{(s,y) \in \R_+^k \times \Z_+^\ell }{ Rs + Py\in S },
\end{equation}
where $k, \ell \in \Z_+$, $R \in \R^{n \times k}$, and $P \in \R^{n \times \ell}$. We allow $k=0$ or $\ell=0$, but not both.
The assumption that $S$ is closed and $0 \not\in S$ implies that $(0,0) \not\in \conv(X_S(R,P))$~\cite[Lemma 2.1]{conforti2014cut}. We search for valid inequalities that separate $(0,0)$ from $X_S(R,P)$.

A \emph{cut-generating (function) pair $(\psi, \pi)$ for $S$} is a pair of functions $\psi, \pi\colon\R^n \to \R$ such that for every $k, \ell \in \Z_+$, $R = (r^1, \dots, r^k) \in \R^{n \times k}$, and $P = (p^1, \dots, p^{\ell}) \in \R^{n \times \ell}$, the inequality
\begin{equation}\label{psi pi ineq-1}
	\sum_{i=1}^k\psi(r^i)s_i + \sum_{j=1}^{\ell}\pi(p^j)y_j \ge 1
\end{equation}
is satisfied by all points $(s,y)\in \conv(X_S(R,P))$.
Note that $(0,0)\in \R^k\times \Z^{\ell}$ does not satisfy~\eqref{psi pi ineq-1}, so the inequality separates $(0,0)$ from $\conv(X_S(R,P))$.
Sometimes we refer to cut-generating pairs as {\em valid cut-generating pairs} or {\em valid pairs} to emphasize that they give valid inequalities of the form~\eqref{psi pi ineq-1}; inequality~\eqref{psi pi ineq-1} is known as a {\em cutting plane} or a {\em cut}. 
The literature studying model~\eqref{def mixed-int set} and cut-generating pairs is extensive. 
We refer the reader to the surveys~\cite{Richard-Dey-2010:50-year-survey,corner_survey,basu2015geometric,basu2016light,basu2016light2} and Chapter 6 of~\cite{conforti2014integer}, and the references within, for an overview of the field.

There is a natural partial order on the set of valid pairs, namely $(\psi', \pi') \leq (\psi, \pi)$ if and only if $\psi' \leq \psi$ and $\pi' \leq \pi$. Since each point $(s,y)\in X_S(R,P)$ is nonnegative, the relation $(\psi', \pi') \leq (\psi, \pi)$ indicates that all cuts obtained from $(\psi, \pi)$ are implied by those obtained from $(\psi', \pi')$.
The minimal elements under this partial order are called {\em minimal valid pairs.} 

The connection between $S$-free sets and cut-generating functions has been instrumental in making cut-generating functions a computational tool for mixed-integer optimization. 
A set $B \subseteq \R^n$ is called a {\em convex $0$-neighborhood} if $B$ is convex and $0 \in \intr(B)$.
If $B$ is a convex $0$-neighborhood and $S \cap \intr(B) = \emptyset$, then $B$ is called an {\em $S$-free convex $0$-neighborhood}. 
If there does not exist a strict superset of $B$ that is also an $S$-free convex $0$-neighborhood, then $B$ is called a {\em maximal} $S$-free convex $0$-neighborhood.
A sublinear\footnote{A function is sublinear if it is convex and subadditive.} function $\gamma :\R^n \to \R$ is called a {\em representation of $B$} if $B = \{r \in \R^n: \gamma(r) \leq 1\}$. 
A convex $0$-neighborhood may have several representations, with the classic {\em gauge} function being one such representation. 
Representations of closed convex $0$-neighborhoods was the main topic of study in~\cite{basu2011sublinear,conforti2013cut}, where it was established that there always exists a {\em smallest} representation $\gamma^*$ for a convex $0$-neighborhood $B$, i.e., $\gamma^* \leq \gamma$ for all representations $\gamma$ of $B$. 

\smallskip

The following recipe provides one way of creating a cut-generating pair:

\smallskip

\begin{tabular}{@{}rl}
1. & Fix a maximal $S$-free convex $0$-neighborhood $B$.\\
2. & Let $\gamma^*$ be the smallest representation of $B$.\\
3. & The pair $(\psi , \pi) = ( \gamma^*, \gamma^*)$ is a cut-generating pair.
\end{tabular}

\smallskip

\noindent 
Unfortunately, this recipe falls short of creating a minimal cut-generating pair because the pair $(\psi , \pi) = ( \gamma^*, \gamma^*)$ is only ``partially minimal''. 
Indeed, one can show that for any other cut-generating pair $(\psi',\pi') \leq (\psi,\pi)$, one must have $\psi'=\psi$. 
However, there may exist another function $\pi' \leq \pi$ such that $(\psi, \pi')$ is also a valid pair. 
This motivates the following definition.
Let $B$ be a maximal $S$-free convex $0$-neighborhood and let $\psi$ be the smallest representation of $B$. 
Then $\pi : \R^n \to \R$ is a {\em lifting} of $\psi$ if $(\psi, \pi)$ is a valid cut-generating pair.
Note that $\psi$ is a lifting of itself. 
The set of all liftings of $\psi$ is partially ordered by pointwise dominance, so one can define {\em minimal liftings}.

The lifting approach to create cut-generating pairs is useful because for some structured sets $S$, the smallest representations of maximal $S$-free convex $0$- neighborhoods have nice, easy-to-compute ``formulas''. Moreover, for some classes of maximal $S$-free convex $0$-neighborhoods, nice ``formulas'' exist for minimal liftings of the smallest representation. For a survey of these ideas, see~\cite{basu2015geometric} and Section 6.3.4 in~\cite{conforti2014integer}.

We say that a function $\psi:\R^n\to \R$ is a {\em valid function for $S$} if $(\psi,\psi)$ is a valid cut-generating pair for $S$. 
The recipe above depends on the observation that the smallest representation of any $S$-free convex $0$-neighborhood (not necessarily maximal) is a valid function for $S$~\cite[Theorem 4.12]{basu2015geometric}. 
However, not all valid functions of $S$ are representations of $S$-free convex $0$-neighborhoods. The notion of a lifting of $\psi$ can be easily extended to any valid function $\psi$ for $S$: $\pi$ is a lifting of $\psi$ if $(\psi, \pi)$ forms a cut-generating pair for $S$. Under pointwise dominance, minimal elements of the set of liftings of a valid function $\psi$ for $S$ will be called minimal liftings of $\psi$.

\subsection{Unique minimal liftings} 
Let $B\subseteq \R^n$ be a maximal $S$-free convex $0$-neighborhood.
A central notion in the study of minimal liftings of the smallest representation $\psi$ of $B$ is the {\em extended lifting region} $R(B)$ defined to be
\begin{equation}\label{eq:lifting-region} R(B) := \{r \in \R^n \colon \pi_1(r) = \pi_2(r) \textrm{ for all minimal liftings }\pi_1, \pi_2 \textrm{ of }\psi\}.\end{equation}

If $R(B) = \R^n$, then $\psi$ has a unique minimal lifting. 
Moreover, nice ``formulas'' for this unique lifting can be derived in terms of $\psi$; see Section 6 of the survey~\cite{basu2015geometric}. 
A large class of maximal $S$-free convex $0$-neighborhoods with this \emph{unique lifting} property has been identified and studied in many recent papers on minimal liftings~\cite{averkov-basu-lifting,bcccz,bck,basu-paat-lifting,dw2008}.
However, the same literature shows that there are many choices for $B$ that satisfy $R(B) \subsetneq \R^n$. 
\textbf{The purpose of this manuscript is to describe minimal valid pairs that arise from such maximal $S$-free convex $0$-neighborhoods, that is, from maximal $S$-free convex $0$-neighborhoods without the unique lifting property.}

Let $p^*\in \R^n$ and assume that $p^* \not\in R(B)$.
This means that there exist two minimal liftings of $\psi$ that disagree on $p^*$. 
When considering a model $X_S(R,P)$ in which $p^*$ is a column of $P$, one would like to develop cuts that have a small coefficient $\pi(p^*)$. To this end, it is of interest to examine the smallest possible value that any minimal lifting of $\psi$ can achieve at $p^*$, which is denoted by
\begin{equation}\label{eqVpsip}
V_{\psi}(p^*):= \inf\{\pi(p^*) : \pi \textrm{ minimal lifting of }\psi\}.
\end{equation}
We aim to find a minimal lifting in the collection 
\begin{equation}\label{def:L-psi} 
\mathcal{L}_{\psi,p^*} := \{\pi:\R^n \to \R: \pi \text{ is a minimal lifting of }\psi \text{ and } \pi(p^*) = V_\psi(p^*)\}. 
\end{equation}
For the setting when $n = 2$ and $S = \Z^2$, Dey and Wolsey~\cite{dw2008} studied $V_{\psi}(p^*)$  and showed that $\mathcal{L}_{\psi,p^*}$ is nonempty.
In general, $\mathcal{L}_{\psi,p^*}$ is nonempty, and we show this in Proposition~\ref{prop:fill-in} in Appendix~\ref{appendix:useful-obs}.

By definition of $V_{\psi}(p^*)$ and the extended lifting region, all $\pi\in \cL_{\psi, p^*}$ agree on $\{p^*\} \cup R(B)$. 
Are there more values on which these liftings agree? 
Analogous to the extended lifting region, we define the \textit{fixing region $\mathcal{F}_{\psi,p^*}$ corresponding to $p^*$} to be the set of points on which all minimal liftings in $\cL_{\psi, p^*}$ agree, that is
\begin{equation}\label{eqFixRegion}
\cF_{\psi,p^*}:= \{p \in \R^n : \pi_1(p) = \pi_2(p) \textrm{ for all }\pi_1, \pi_2 \in \mathcal{L}_{\psi,p^*}\}.
\end{equation}

If $\cF_{\psi,p^*}  = \R^n$, then there exists a unique lifting in $\mathcal{L}_{\psi,p^*}$.
In other words, after finding the optimal lifting coefficient $V_\psi(p^*)$ for $p^*$, the lifting coefficients for all other vectors are uniquely determined for all minimal liftings that assign $V_\psi(p^*)$ to the vector $p^*$. 
If there exists a $p^*$ such that $\cF_{\psi, p^*} = \R^n$, then we say that $\psi$ and the underlying set $B$ are {\em one point fixable}. 

Using the fixing region, the recipe provided above can be modified to create minimal cut-generating pairs.

\begin{tabular}{@{}rl}
1. & Fix a maximal $S$-free convex $0$-neighborhood $B$ that is one point fixable.\\
2. & Let $\psi$ be the smallest representation of $B$.\\
3. & Find $p^*\in \R^n$ such that $\mathcal{F}_{\psi,p^*}  = \R^n$.\\
4. & Then $\mathcal{L}_{\psi,p^*} = \{\pi\}$ and the pair $(\psi , \pi)$ is a minimal cut-generating pair.
\end{tabular}

\smallskip

In this paper, we study the structure of the fixing region and one-point fixability.
What is a good description of the fixing region? 
How does the fixing region depend on $p^*$? 
How much does the fixing region cover? 
We explore questions such as these.
Our work is \textcolor{black}{motivated} by Section 7 of~\cite{dw2008}, which initiated the study of this problem. 

\subsection{Statement of results}
To state our results, we need the set
\begin{equation*}
W_S := \{w\in \R^n : s+\lambda w\in S~, \forall s\in S, \forall \lambda\in \Z\}.
\end{equation*}
The importance of $W_S$ is that any minimal lifting $\pi$ of a valid function $\psi$ satisfies $\pi(r+w) = \pi(r)$ for all $r\in \R^n$ and $w\in W_S$ (see Proposition~\ref{prop:easy-partial}).

Let $B \subseteq \R^n$ be a maximal $S$-free convex $0$-neighborhood and let $\psi$ be the corresponding smallest representation.

\smallskip

\begin{enumerate}[leftmargin=*]
\item Let $p^*\in \R^n$.
In Theorem~\ref{thm:lifting-fixing}, we use the structure of $B$ to identify a nonempty set $\cX(B, p^*)\subseteq \R^n$ such that $R(B)\subsetneq \cX(B, p^*) + W_S \subseteq  \cF_{\psi,p^*}$.
It is not known if this inner approximation of the fixing region is always equal to $\cF_{\psi,p^*}$.

\smallskip

\item
In Proposition~\ref{prop:covering_prop_type3}, we use the inner approximation in Theorem~\ref{thm:lifting-fixing} to show that certain Type 3 triangles are one point fixable. 
As a corollary, in Proposition~\ref{prop: mix_pyramid_lattice_free} we show that Type 3 triangles resulting from the so-called mixing set are one point fixable. This also follows from~\cite[Theorem 5]{dw2008}; we use different, more geometric techniques. See~\cite{dw2,mixingSet} for more on the mixing set.

\smallskip

\item Theorem~\ref{thm:fixing_covering_trans} says if our inner approximation $\cX(B, p^*)+W_S$ of $\cF_{\psi,p^*}$ equals $\R^n$ (implying that $B$ is one point fixable), then the $(S+t)$-free convex $0$-neighborhood $B+t$ is one point fixable for any $t\in \R^n$ such that $B+t$ is a $0$-neighborhood. 
In other words, one point fixability is preserved under translations. 
If an $S$-free $0$-neighborhood is used to derive cuts around a basic feasible solution of a mixed-integer linear program, then, by using this translation invariance, these cuts can be transformed to cuts around a different basic feasible solution.
A more detailed discussion of this point is provided in~\cite{basu-paat-lifting} and~\cite{basu2015geometric}.
Theorem~\ref{thm:fixing_covering_trans} is in Subsection~\ref{subsecTranslation}.

\smallskip

\item In Section~\ref{ssecPcgf}, we develop a theory of {\em partial cut-generating pairs}, which are cut-generating pairs that are only defined on subsets of $\R^n$.
Partial cut-generating pairs, which were first developed in this paper, have been subsequently used in~\cite{BCDSPInfinite2017} to prove structural results about the infinite models in integer programming. One way to think of the results in Section~\ref{ssecPcgf} is that they are analogous to classic ``lifting" results like Hahn-Banach theorems in analysis~\cite{conway2013course}, and ``lifting" valid inequalities from faces of a polytope to the full polytope (see, e.g., Section 7.2 in~\cite{conforti2014integer}).



\end{enumerate}


\section{Partial cut-generating functions}\label{ssecPcgf}

We denote the columns of a matrix $A$ by $\col(A)$.
For a set $X$ and any $d \in \N$, $X^d$ will denote the $d$-wise Cartesian product of $X$ with itself.
 Let $\mathcal{R}, \mathcal{P} \subseteq \R^n$, $\psi: \cR \to \R$, and $\pi:\cP \to \R$. 
We define $(\psi, \pi)$ to be a \emph{valid pair for $(S,\mathcal{R}, \mathcal{P})$} if for every $k, \ell \in \Z_+$, $R \in \cR^k$, and $P \in \cP^{\ell}$,
the inequality
\begin{equation}
	\label{psi pi ineq}
	\sum_{r\in \col(R)}\psi(r)s_r + \sum_{p\in \col(P)}\pi(p)y_p \ge 1
\end{equation}
is satisfied by all points $(s,y) \in X_S(R,P)$. 
Here, $s_r$ denotes the continuous variable associated with $r \in \col(R)$ and $y_p$ denotes the integer variable associated with $p\in \col(P)$. 
The concepts of a valid function $\psi:\cR \to \R$ for $(S, \cR)$ and a minimal valid pair for $(S, \cR, \cP)$ are defined analogously to the case $\cR = \R^n$ and $\cP = \R^n$. 
For $\cP \subseteq \R^n$, we say $\pi: \cP \to \R$ is a \emph{lifting of a valid function $\psi$ for $(S, \cR)$}, if $(\psi, \pi)$ is a valid pair for $(S,\mathcal{R}, \mathcal{P})$. 
The concept of a minimal lifting of $\psi$ is analogously defined. 
When $\cR$ and $\cP$ are strict subsets of $\R^n$, we refer to $\psi$ as a \textit{partial cut-generating function} and $(\psi, \pi)$ as a \textit{partial cut-generating pair}. 

Using this terminology, valid pairs for $S$ defined in Section~\ref{sec:intro} become valid pairs for $(S, \R^n, \R^n)$ and valid functions for $S$ become valid functions for $(S, \R^n)$. 
In the remaining text, we will be careful about explicitly stating $\mathcal{R}$ and $\mathcal{P}$ whenever we speak about valid functions or pairs.

Minimal cut-generating pairs for $(S, \R^n, \R^n)$ satisfy certain structural properties. 
The next proposition shows that some of these results also hold for partial cut-generating pairs, and setting $\cR = \cP = \R^n$ recovers the setting of cut-generating pairs.
Similarly to the translation set $W_S$ for classic cut-generating pairs, define
\begin{equation*}
W^+_S := \{w\in \R^n : s+\lambda w\in S~, \forall s\in S, \forall \lambda\in \Z_+\}
\end{equation*}
for partial cut-generating pairs. 
Note that $W_S = W^+_S \cap (-W^+_S)$.

\begin{proposition}\label{prop:easy-partial}
Let $S \subseteq \R^n\setminus\{0\}$ be a closed set. 
Let $\mathcal{R}, \mathcal{P} \subseteq \R^n$ and $\psi: \cR \to \R$  be a valid function for $(S, \cR)$. 
\begin{itemize}[leftmargin=*]
\item[(a)] For any minimal lifting $\pi$ of $\psi$, $\pi(p) \leq \pi(p+w)$ for all $p \in \cP$ and $w \in W^+_S$ such that $p + w \in \cP$. 
So, $\pi(p) = \pi(p+w)$ for all $p \in \cP$ and $w \in {W}_S$ such that $p+w\in \cP$.
\item[(b)] Define $\psi^*: \cR \to \R$ to be
\[
\psi^*(r) := \inf\{\psi(r+w): w \in W^+_S \textrm{ such that } r+w \in \cR\}.
\]
Then $(\psi, \psi^*)$ is a valid partial cut-generating pair for $(S, \cR, \cR)$.
\item[(c)] If $\cR = \cP$, then every minimal lifting $\pi$ of $\psi$ satisfies $\pi \leq \psi^*$.
\end{itemize}

\end{proposition}

\begin{proof}
Let $\cK \subseteq \R^n$ and take $\sigma :\cK \to \R$ to be a (not necessarily minimal) lifting of $\psi$.
Thus, $(\psi, \sigma)$ is a valid pair for $(S, \cR, \cK)$.
Define $\sigma^*:\cK \to \R$ to be 
\[
\sigma^*(p) := \inf_{w\in W^+_S}\bigg\{\sigma(p+w): ~ p+w \in \cK \bigg\}.
\]
First, we show that $(\psi, \sigma^*)$ is a valid pair for $(S, \cR, \cK)$.

Let $k,\ell\in \Z_+$, $R \in \cR^k$, $P \in \cK^{\ell}$, and $(s,y) \in X_S(R,P)$. 
Let $W \in \R^{n \times \ell}$ be any matrix with $\col(W)\subseteq W^+_S$ such that $P+W\in \cK^{\ell}$. 
Let $(\bar s, \bar y) \in \R^k_+\times \Z^{\ell}_+$ be constructed as follows: $\bar s_r = s_r$ for each $r\in \col(R)$ and $\bar{y}_{p+w} = y_{p}$ for each $p+w \in \col(P+W)$. 
Since $\bar w \in W^+_S$ by definition of $W$, it follows that $R\bar s + (P+W)\bar y = Rs + Py + \bar w\in S$. 
Thus, since $(\psi, \sigma)$ is a valid pair for $(S, \cR, \cK)$,
$$\sum_{r\in \col(R)} \psi(r)\overline{s}_{r} \quad + \sum_{p+w\in \col(P+W)}\sigma(p+w)\overline{y}_{p+w} \quad \geq 1.$$ 
The above holds for all matrices $W \in \R^{n \times \ell}$ whose columns are in $W^+_S$ and $P+W\in \cK^\ell$. 
Taking an infimum over all such $W$ gives 
\begin{align*}
& \sum_{r\in \col(R)}\psi(r)s_{r} + \sum_{p\in \col(P)}\sigma^*(p)y_{p}\\
=& \sum_{r\in \col(R)}\psi(r)s_{r} + \inf_{W}\bigg\{ \sum_{p+w\in \col(P+W)}\sigma(p+w)y_{p}\bigg\}\\
=&  \inf_{W}\bigg\{\sum_{r\in \col(R)}\psi(r)s_{r}  \quad + \sum_{p+w\in \col(P+W)}\sigma(p+w)\overline{y}_{p+w} \bigg\}\\
\geq& 1.
\end{align*}
Thus, $(\psi, \sigma^*)$ is also a valid pair for $(S, \cR, \cK)$.
Setting $\sigma = \psi$ and $\cK = \cR$ gives $(b)$.

Let $\pi$ be a minimal lifting of $\psi$.
Set $\sigma = \pi$ and $\cK = \cP$.
Since $\pi^* \leq \pi$ and $\pi$ is minimal, we obtain $\pi^* = \pi$.
Hence, $\pi(p) = \pi^*(p) \leq \pi(p+w)$ for all $p \in \cP$ and $w \in W^+_S$ such that $p+w\in \cP$.
This proves $(a)$.

Finally, assume that $\cP = \cR $.
Since $\pi$ is a minimal lifting of $\psi$, $\pi(r) \leq \psi(r)$ for all $r \in \cR$.
By $(a)$, $\pi(p) \leq \pi(p+w) \leq \psi(p+w)$ for all $p \in \cP$ and $w \in W^+_S$ such that $p+w\in \cP = \cR$.
Taking an infimum over all such $w \in W^+_S$, we obtain $(c)$.
\end{proof}

Theorem~\ref{thm:subadditive} follows from standard calculations involving cut-generating functions, so the proof is omitted.


\begin{theorem}\label{thm:subadditive} 
Let $(\psi, \pi)$ be a minimal valid pair for $(S, \cR, \cP)$.  
Then $\psi$ and $\pi$ are both subadditive over $\cR$ and $\cP$, respectively, i.e., $\psi(r + r') \leq \psi(r) + \psi(r')$ for all $r, r'\in \cR$ such that $r + r' \in \cR$, and $\pi(p + p') \leq \pi(p) + \pi(p')$ for all $p, p'\in \cP$ such that $p + p' \in \cP$. 
Also, $\psi$ is positively homogeneous over $\cR$, i.e., for all $r \in \cR$ and $\lambda >0$ such that $\lambda r \in \cR$, we have $\psi(\lambda r) = \lambda\psi(r)$.
\end{theorem}

Given $\mathcal{R} \subseteq \mathcal{R'}\subseteq \R^n, \mathcal{P}\subseteq \mathcal{P}' \subseteq \R^n$, and a valid pair $(\psi, \pi)$ for $(S,\mathcal{R}, \mathcal{P})$, a natural question is that of extension: do there always exist functions $\psi', \pi'$ such that $(\psi', \pi')$ is valid for $(S, \cR', \cP')$ and $\psi', \pi'$ are extensions of $\psi,\pi$, i.e., they coincide on $\cR$ and $\cP$ respectively?
The answer to the question is `\emph{no}', in general. 
Indeed, choosing $\mathcal{R} = \emptyset$ and $\mathcal{P} = \R^n$, we obtain Gomory and Johnson's pure integer model, where the discontinuous valid functions $\pi$ cannot be appended to any $\psi$ to give a valid pair for the full mixed-integer model (see~\cite{dey1}). 
On the positive side, the next result gives a sufficient condition for when partial cut-generating pairs can be extended.

For a set $X \subseteq \R^n$, we use $\cone(X)$ to denote the convex cone generated by $X$.

\begin{theorem}\label{thm:extension}
Let $\mathcal{R} \subseteq \mathcal{R'}\subseteq \R^n, \mathcal{P}\subseteq \mathcal{P}' \subseteq \R^n$ and $(\psi, \pi)$ be a valid pair for $(S,\mathcal{R}, \mathcal{P})$. 
If $\cR', \cP' \subseteq \cone(\cR)$, then there exist functions $\psi':\cR'\to \R$, $\pi':\cP' \to \R$ such that $(\psi', \pi')$ is a minimal valid pair for $(S, \cR', \cP')$ and $(\psi', \pi') \le (\psi, \pi)$ on $\cR \times \cP$.
\end{theorem}

\begin{proof} 
For $r'\in \cR'$, define 
\begin{equation*}
\nu_{\psi}(r') := \inf\left\{\sum_{r\in \cR}\psi(r)h(r) :
\begin{array}{l}
  \displaystyle r' =\sum_{r\in \cR}rh(r) \text{ and}\\
  h: \cR \to \R_+ \text{ has finite support}
\end{array}
\right\}.
\end{equation*}
Similarly, for $p'\in \cP'$ define
\begin{equation*}
\nu_{\pi}(p'):= \inf\left\{\sum_{r\in \cR}\psi(r)h(r) +\sum_{p\in \cP}\pi(p)g(p):
\begin{array}{l}
\displaystyle p' = \sum_{r\in \cR}rh(r)+\sum_{p\in \cP}pg(p), \\
h: \cR \to \R_+ ~ \text{has finite support and}\\
g: \cP\to \Z_+  ~ \text{has finite support}
\end{array}
\right\}.
\end{equation*}
Since $\cR', \cP' \subseteq \cone(\cR)$, the infima defining $\nu_{\psi}(r')$ and $\nu_{\pi}(p')$ are over nonempty sets.
Thus, $\nu_{\psi}(r') \in [-\infty, \infty)$ for all $r'\in \cR'$ and $\nu_{\pi}(p') \in [-\infty, \infty)$ for all $p'\in \cP'$.

\smallskip

Define the functions $\tilde{\psi}:\cR'\to \R$ and $\tilde{\pi}:\cP'\to \R$ to be
\begin{equation*}
\tilde{\psi}(r') := 
\begin{cases} \nu_{\psi}(r')~&\text{if }\nu_{\psi}(r') > -\infty\\
\psi(r')~&\text{if }\nu_{\psi}(r') = -\infty\text{ and }r'\in \cR\\
0~&\text{otherwise,}
\end{cases}
\end{equation*}
and
\begin{equation*}
\tilde{\pi}(p') := 
\begin{cases} \nu_{\pi}(p')~&\text{if }\nu_{\pi}(p') > -\infty\\
\pi(p')~&\text{if }\nu_{\pi}(p') = -\infty\text{ and }p'\in \cP
\\0~&\text{otherwise}.
\end{cases}
\end{equation*}

Let $r\in \cR$ and define $h:\cR \to \R_+$ to be $h(r) = 1$ and $h(r')=0$ for all $r'\in \cR\setminus \{r\}$.
If $\nu_{\psi}(r)  = -\infty$, then $\tilde{\psi}(r)\leq \psi(r)$.
If $\nu_{\psi}(r)  = -\infty$, then 
\[
\tilde{\psi}(r) = \nu_{\psi}(r) \le \sum_{r\in \cR} \psi(r)h(r) = \psi(r).
\]
Hence, $\tilde{\psi}(r)\leq \psi(r)$ for every $r\in \cR$.
Similarly, $\tilde{\pi}(p)\leq \pi(p)$ for every $p\in \cP$. 
Therefore, $(\tilde{\psi}, \tilde{\pi}) \le (\psi, \pi)$ on $\cR \times \cP$. 
Zorn's Lemma implies that any valid pair is pointwise larger than some minimal valid pair (see, for example, Proposition A.1. in~\cite{basu-paat-lifting}), so it is sufficient to show that $(\tilde{\psi}, \tilde{\pi})$ is valid for $(S, \cR', \cP')$. 

Let $R'$ and $P'$ be matrices with columns in $\cR'$ and $\cP'$, respectively.
Consider $(s',y') \in X_S(R', P')$ and let $\epsilon>0$. 
Let $r' \in \col(R') \subseteq \cR' \subseteq \cone(\cR)$. 
By the definition of $\tilde{\psi}$, there exists a function $h_{r'} :\cR \to \R_+$ with finite support such that
\begin{equation*}
r' = \sum_{r\in \cR} rh_{r'}(r)~~~\text{and}~~~\tilde{\psi}(r') > \left(\sum_{r\in \cR}\psi(r)h_{r'}(r)\right) - \epsilon.
\end{equation*}
Similarly, for each $p' \in \col(P')$, there exist functions $h_{p'}:\cR\to \R_+$ and $g_{p'}: \cP\to \Z_+$, both with finite support, such that
\begin{equation*}
p' = \sum_{r\in \cR} rh_{p'}(r)+\sum_{p\in \cP} pg_{p'}(p)~~~\text{and}~~~\tilde{\pi}(r) > \left(\sum_{r\in \cR}\psi(r)h_{p'}(r)+\sum_{p\in \cP}\pi(p)g_{p'}(p)\right) - \epsilon.
\end{equation*}

Define the matrix $R \in \R^{n \times \mathopen{|}\col(R)|}$ to have columns
\[
\col(R) := \bigcup_{r'\in R'} \text{support}(h_{r'}) ~ \cup ~ \bigcup_{p'\in P'} \text{support}(g_{p'}),
\]
and the matrix $P \in \R^{n\times \mathopen{|}\col(P)|}$ to have columns
\[
\col(P) := \bigcup_{r'\in R'} \text{support}(g_{p'}).
\]
Define $(\tilde{s},\tilde{y})\in \R^{\mathopen{|}\col(R)|}_+\times \Z^{\mathopen{|}\col(P)|}_+$ component-wise to be 
\begin{equation*}
\begin{array}{rcll}
\tilde{s}_r &:=&\displaystyle \sum_{r'\in R'}h_{r'}(r)s'_{r'}+\sum_{\substack{p'\in P'}}h_{p'}(r)y'_{p'} & \forall ~ r\in \col(R), \text{ and }\\[.5 cm]
\tilde{y}_p &:=&\displaystyle  \sum_{\substack{p'\in P'}}g_{p'}(p)y'_{p'} & \forall ~ p\in \col(P).
\end{array}
\end{equation*}
Using the fact that $(s', y') \in X_S(R', P')$ and the definitions of $\tilde{s}$ and $\tilde{y}$, it follows that $R\tilde{s} + P\tilde{y} = R's' + P'y' \in S$.
This implies that $(\tilde{s}, \tilde{y})\in X_S(R, P)$. 
Set $M :=  \sum_{r'\in R'}s'_{r'}+\sum_{p'\in P'}y'_{p'}$.
The value $M$ is a constant because $s'$ and $y'$ are fixed. 
Since $(\psi, \pi)$ is valid for $(S, \cR, \cP)$, we see that

\begin{align*}
&\sum_{r'\in \cR'} \tilde{\psi}(r')s'_{r'}+\sum_{p'\in \cP'}\tilde{\pi}(p')y'_{p'}\\
\geq &\sum_{r'\in \cR'}\left[ \sum_{r\in \cR}\psi(r)h_{r'}(r)-\epsilon\right]s'_{r'}+\sum_{p'\in \cP'}\left[\sum_{r\in \cR}\psi(r)h_{p'}(r) +\sum_{p\in \cP}
\pi(p)g_{p'}(p)-\epsilon\right]y'_{p'}\\
= & \sum_{\substack{r\in \cR\\r'\in \cR'}}\psi(r)h_{r'}(r)s'_{r'}+\sum_{\substack{r\in \cR\\p'\in \cP'}}\psi(r)h_{p'}(r)y'_{p'}+\sum_{\substack{p\in \cP\\p'\in \cP'}}\pi(p)g_{p'}(p)y'_{p'} -\epsilon M\\
= & \sum_{r\in \cR}\psi(r)\tilde{s}_r +\sum_{p\in \cP} \pi(p)\tilde{y}_p - \epsilon M\\
\geq &1-\epsilon M.
\end{align*}
Letting $\epsilon \to 0$ yields
\[
\sum_{r'\in \cR'} \tilde{\psi}(r')s'_{r'}+\sum_{p'\in \cP'}\tilde{\pi}(p')y'_{p'} \ge 1.
\]
Hence, $(\tilde{\psi}, \tilde{\pi})$ is a valid pair for $(S, \cR', \cP')$.
\end{proof}


\section{The fixing region for truncated affine lattices}\label{section:fixing_region}

We now examine the fixing region $\mathcal{F}_{\psi,p^*}$ for a valid function $\psi$ and different choices of $p^*$.
For the rest of the paper, we assume that $S = (b + \Z^n) \cap P$, where $b \in \R^n\setminus \Z^n$ and $P \subseteq \R^n$ is a rational polyhedron. 
These $S$ were called {\em polyhedrally-truncated affine lattices} in~\cite{basu-paat-lifting}.

Let $p^* \in \R^n$ and recall $\cL_{\psi, p^*}$ from~\eqref{def:L-psi}.
One way of finding a minimal lifting of $\psi$ is to find a function $\pi \in \cL_{\psi, p^*}$.
Proposition~\ref{prop:fill-in} in Appendix~\ref{appendix:useful-obs} shows that $\cL_{\psi, p^*}$ is nonempty.
An important ingredient for finding $\pi \in \cL_{\psi, p^*}$ is the value $V_\psi(p^*)$ from~\eqref{eqVpsip}.
In~\cite{dw2008}, Dey and Wolsey gave the following algebraic formula for $V_\psi(p^*)$:
\begin{equation}\label{eq:V_psi}
V_\psi(p^*) = \sup_{w \in \R^n, N \in \N}\bigg\{\frac{1-\psi(w)}{N}: w + Np^* \in S\bigg\}.\end{equation} 

A more geometric description of $V_\psi(p^*)$ was given in~\cite{bcccz}. 
Let $B\subseteq \R^n$ be a maximal $S$-free convex $0$-neighborhood.
Because $S$ is a truncated affine lattice, $B$ is a polyhedron of the form
\begin{equation}\label{eqBDefn}
B= \{r \in \R^n : a^i\cdot r \leq 1 ~\forall~ i \in I\},
\end{equation}
where $I$ is a finite set indexing the facets of $B$ ~\cite{bccz2,moran2011maximal}.
Also, the smallest representation of $B$ is
\begin{equation}\label{eq:psi-from-B}
\psi_B(r) = \max_{i\in I}a^i\cdot r.
\end{equation}
If $B$ is any $S$-free $0$-neighborhood of the form~\eqref{eqBDefn}, even if it is not maximal, then~\eqref{eq:psi-from-B} gives a valid function for $(S, \R^n)$. This fact will be used later.

For $\lambda > 0$, define $\pyr(B, \lambda, p^*)$ to be the pyramid in $\R^n \times \R_+$ with $\frac1\lambda (p^*, 1)$ as the apex and $B \times \{0\}$ as the base, i.e.
\begin{equation}\label{eq:Blambda}
\pyr(B, \lambda, p^*) := \{(r, r_{n+1})\in \R^n\times \R_+
: a^i\cdot r + (\lambda - a^i\cdot p^*)r_{n+1} \leq 1 ~ \forall ~ i\in I\}.
\end{equation}

The following was shown in~\cite[Theorem 11]{bcccz}.
\begin{proposition}\label{prop:geometric-V}
Let $B\subseteq \R^n$ be a maximal $S$-free convex $0$-neighborhood and let $\psi:= \psi_B:\R^n\to \R$ be obtained from $B$ using~\eqref{eq:psi-from-B}.
If $p^* \in \R^n$, then 
\[
V_{\psi}(p^*) = \inf\left\{\lambda > 0 : 
\pyr(B, \lambda, p^*) \textrm{ is } (S \times \Z)\textrm{-free}
\right\}.
\]
\end{proposition} 

In~\cite{bcccz}, the authors studied a variant of $\pyr(B, \lambda, p^*)$ in which $r_{n+1}$ was not constrained to be nonnegative.
Their characterization of $V_{\psi}(p^*)$ in Proposition~\ref{prop:geometric-V} is simply given by \emph{$\pyr(B, \lambda, p^*)$ is $(S\times \Z_+)$-free}.
However, their proof also holds for the current definition of $\pyr(B, \lambda, p^*)$ and Proposition~\ref{prop:geometric-V}.

\subsection{A geometric perspective on $\mathcal L_{\psi,p^*}$}

The main tool for our geometric approach to understanding $\mathcal L_{\psi,p^*}$ is the polyhedron $\pyr(B, V_\psi(p^*), p^*)$ from~\eqref{eq:Blambda}. 

Let $B\subseteq \R^n$ be a maximal $S$-free convex $0$-neighborhood of the form~\eqref{eqBDefn}, $\psi:=\psi_B:\R^n\to \R$ be the valid function for $(S, \R^n)$ obtained from $B$ using~\eqref{eq:psi-from-B}, and $p^*\in \R^n$.
A point $(\bar{x}, \bar{x}_{n+1}) \in S \times \Z_+$ with $\bar{x}_{n+1} \geq 1$ such that $\pyr(B, V_\psi(p^*), p^*)$ contains $(\bar{x}, \bar{x}_{n+1})$ is called a {\em blocking point} for $\pyr(B, V_\psi(p^*), p^*)$.

It was shown in~\cite[Theorem 11]{bcccz} that there is at least one blocking point for $\pyr(B, V_\psi(p^*), p^*)$ for every $p^* \in \R^n$. 
Lemma~\ref{lem:blocking-maximizer} relates the algebraic formula~\eqref{eq:V_psi} for $V_{\psi}(p^*)$ and the important geometric notion of a blocking point for $\pyr(B, V_\psi(p^*), p^*)$. 
Since blocking points always exist, Lemma~\ref{lem:blocking-maximizer} implies that the supremum in~\eqref{eq:V_psi} is actually a maximum and the infimum in Proposition~\ref{prop:geometric-V} is actually a minimum.

\begin{lemma}\label{lem:blocking-maximizer}
Let $B\subseteq \R^n$ be a maximal $S$-free convex $0$-neighborhood of the form \eqref{eqBDefn}.
Let $\psi:\R^n\to \R$ be the valid function for $(S, \R^n)$ obtained from $B$ using~\eqref{eq:psi-from-B}. 
If $(\bar{x}, \bar{x}_{n+1}) \in S \times \Z_+$ is a blocking point of $\pyr(B, V_\psi(p^*), p^*)$, then
\begin{equation*}
(\bar{x} - \bar{x}_{n+1}p^*, \bar {x}_{n+1}) \in \argmax_{w \in \R^n, ~N \in \N}\bigg\{\frac{1-\psi(w)}{N}: w + Np^* \in S\bigg\}.
\end{equation*}
Conversely, if $(w,N)\in \R^n \times \N$ is a maximizer of~\eqref{eq:V_psi}, then $(w + Np^*, N)$ is a blocking point of $\pyr(B, V_\psi(p^*), p^*)$.
\end{lemma}
\begin{proof} 
By~\eqref{eq:Blambda}, $(\bar{x}, \bar{x}_{n+1})$ is a blocking point of $\pyr(B, V_\psi(p^*), p^*)$ if and only if 
\[
a^i\cdot \bar{x} + (V_\psi(p^*) - a^i\cdot p^*)\bar{x}_{n+1} \leq 1 \quad \forall ~ i\in I,
\]
and there exists some $i^* \in I$ such that $a^{i^*}\cdot \bar{x} + (V_\psi(p^*) - a^{i^*}\cdot p^*)\bar{x}_{n+1} = 1$. 
So, $(\bar{x}, \bar{x}_{n+1})$ is a blocking point of $\pyr(B, V_\psi(p^*), p^*)$ if and only if $\bar{x}_{n+1}V_\psi(p^*) + \max_{i \in I} \{a^i\cdot (\bar{x} - \bar{x}_{n+1}p^*)\} = 1$.
By~\eqref{eq:V_psi}, the latter condition holds if and only if 
\[
V_\psi(p^*) = \frac{1 - \psi(\bar{x} - \bar{x}_{n+1}p^*)}{\bar{x}_{n+1}} = \sup_{w \in \R^n, N \in \N}\bigg\{\frac{1-\psi(w)}{N}: w + Np^* \in S\bigg\}.
\] 
This completes the proof.
\end{proof}

\subsection{A universal upper bound}\label{subsecUniUpper} In order to determine what vectors are in $\cF_{\psi, p^*}$, we first show an upper bound on the value of minimal liftings of $\psi$ and then show that this upper bound is tight.
Theorem~\ref{thm:uni upperbound} gives an upper bound using the function $\psi^*_{[p^*, B]} : \R^n \times \R_+ \to \R$ defined by

\begin{equation}\label{eq:ex_trivial_lifting}
\psi^*_{[p^*,B]}((r, r_{n+1})) := \inf\{\psi_{\pyr(B, V_\psi(p^*), p^*)}((r, r_{n+1})+(w,z) ) : (w,z) \in W^+_{S\times \Z_+}\}.
\end{equation}

In~\eqref{eq:ex_trivial_lifting}, $\pyr(B, V_\psi(p^*), p^*)$ is the set from~\eqref{eq:Blambda}, and $\psi_{\pyr(B, V_\psi(p^*), p^*)}$ is obtained from~\eqref{eq:psi-from-B} using $\pyr(B, V_\psi(p^*), p^*)$ written as $\pyr(B, V_\psi(p^*), p^*)=\{(r, r_{n+1}) \in \R^{n}\times \R_+: \bar{a}^i\cdot r \leq 1 ~ \forall ~ i \in I\}$.
We caution the reader that the formula~\eqref{eq:psi-from-B} was introduced for $B$ that contain $0$ in the interior. However, the formula is a well-defined one, even if $0$ lies on the boundary, as is the case for $\pyr(B, V_\psi(p^*), p^*)$. While there is an interpretation of $\psi_{\pyr(B, V_\psi(p^*), p^*)}$ as a cut-generating function for $(S \times \Z_+, \R^n \times \R_+)$, it is not important in what follows. What is important is Theorem~\ref{thm:uni upperbound}, which shows that the restriction of $\psi^*_{[p^*, B]}$ to $\R^n \textcolor{black}{\times \{0\}}$ is a universal upper bound for all minimal liftings $\pi\in \mathcal{L}_{\psi,p^*}$. We view~\eqref{eq:ex_trivial_lifting} as a formula via~\eqref{eq:psi-from-B} applied to $\pyr(B, V_\psi(p^*), p^*)$. The following technical lemma will be useful for establishing this upper bound. 

\begin{lemma}\label{lem:additive} 
Let $B$ be a convex $0$-neighborhood of the form~\eqref{eqBDefn}. 
Let $p^* \in \R^n$ and $\lambda > 0$. 
For $(\textcolor{black}{\bar r}, \textcolor{black}{\bar r_{n+1}}) \in \R^{n}\times \R_+$ and $\mu \geq 0$, define $r' := (\bar r, \bar r_{n+1}) - \mu(p^*, 1)$. 
Then $\psi_{\pyr(B, \lambda, p^*)}((\textcolor{black}{\bar r}, \textcolor{black}{\bar r_{n+1}})) = \psi_{\pyr(B, \lambda, p^*)}(r') + \mu\psi_{\pyr(B, \lambda, p^*)}((p^*,1)).$
\end{lemma}

\begin{proof}~
First, we show 
\begin{align*}
  \argmax_{i \in I} \{a^i\cdot \bar r + (\lambda - a^i\cdot p^*)\bar r_{n+1}\} 
= & \argmax_{i \in I} \{a^i\cdot (\bar r -\bar r_{n+1}p^*)\}\\
= & \argmax_{i\in I} \{(a^i, (\lambda - a^i\cdot p^*))\cdot r'\}.
\end{align*}
The first and second terms are equal since $\lambda \bar{r}_{n+1}$ is a constant, while the first and the third terms are equal because, for every $i \in I$
$$ a^i\cdot \bar r + (\lambda - a^i\cdot p^*)\bar r_{n+1} = a^i\cdot (\bar r - \mu p^*) + (\lambda - a^i\cdot p^*)(\bar r_{n+1} - \mu) +\lambda\mu = (a^i, (\lambda - a^i\cdot p^*))\cdot r' + \lambda\mu.$$

For $i^* \textcolor{black}{\in} \argmax_{i \in I} \{a^i\cdot \bar r + (\lambda - a^i\cdot p^*)\bar r_{n+1}\}$,
$$
\begin{array}{rcl}
\psi_{\pyr(B, \lambda, p^*)}((\bar r,\bar r_{n+1})) 
&= &a^{i^*}\cdot \bar r + (\lambda - a^{i^*}\cdot p^*)\bar r_{n+1} \\ 
& = & (a^{i^*}, (\lambda - a^{i^*}\cdot p^*))\cdot r' + (a^{i^*}, (\lambda - a^{i^*}\cdot p^*))\cdot \mu(p^*,1) \\ 
& = & \psi_{\pyr(B, \lambda, p^*)}(r') + \mu\psi_{\pyr(B, \lambda, p^*)}((p^*,1)). 
\end{array}
$$
The last equation holds because $(a^{i^*}, (\lambda - a^{i^*}\cdot p^*))\cdot (p^*,1) = \lambda = \psi_{\pyr(B, \lambda, p^*)}((p^*,1))$.
\end{proof}

\begin{theorem}\label{thm:uni upperbound}
Let $B$ be a maximal $S$-free $0$-neighborhood of the form~\eqref{eqBDefn} and $\psi$ be the valid function for $(S, \R^n)$ obtained from $B$ using~\eqref{eq:psi-from-B}.
Let $p^* \in \R^n$ and consider $\psi^*_{[p^*, B]}$ defined in~\eqref{eq:ex_trivial_lifting}.
For $\pi\in \mathcal{L}_{\psi,p^*}$ and $p\in \R^n$, $\pi(p) \leq \psi^*_{[p^*, B]}((p,0))$. 
\end{theorem}

\begin{proof}  
To reduce notation in this proof, set 
\(
\Delta := \pyr(B, V_{\psi}(p^*), p^*).
\)
Let $\pi\in \mathcal{L}_{\psi,p^*}$.
Define $\cR := (\R^n\times \{0\}) \cup \{(p^*,1)\}$ and $\cP := \R^n \times \{0\}$.
Using $(\psi, \pi)$, which is a valid pair for $(S, \R^n, \R^n)$, we will create functions $\hat{\psi}:\R^{n+1}\to \R$ and $\hat{\pi} : \R^{n+1} \to \R$ such that $(\hat{\psi}, \hat{\pi})$ is a valid pair for $(S\times \Z_+, \cR, \cP)$.
Since $\R^n \times \R_+ \subseteq \cone(\cR)$, we will be able to apply Theorem \ref{thm:extension} to obtain a minimal valid pair $(\psi', \pi')$ for $(S, \R^n\times \R_+, \R^n \times \R_+)$ that equals $(\psi, \pi)$ on $\R^n\times \R^n$ and satisfies $(\psi', \pi') \le (\psi_\Delta, \psi^*_{[p^*, B]})$ when restricted to $(\R^n\times \{0\}) \times (\R^n\times \{0\})$.

Define $\hat \psi: \cR \to \R$ by $\hat \psi ((r,0)) = \psi(r)$ for all $r \in \R^n$ and $\hat\psi((p^*,1)) = V_\psi(p^*)$. Define $\hat\pi: \cP \to \R$ by $\hat\pi((p,0)) = \pi(p)$ for all $p\in \R^n$.
 
\begin{CLAIM}\label{claim:hatpsivalid}
$(\hat \psi, \hat \pi)$ is valid for $(S\times \Z_+, \cR, \cP)$. 
\end{CLAIM}

\begin{cpf}
Consider matrices $R\in \R^{(n+1)\times k}$ and $P\in \R^{(n+1)\times \ell}$ with columns in $\cR$ and $\cP$, respectively.
Let $(\bar s, \bar y) \in X_{S\times \Z_+}(R,P)$. 
Using two cases, we show that $(\hat \psi, \hat \pi)$ and $(\bar s, \bar y)$ satisfy~\eqref{psi pi ineq-1}.
First, assume that $(p^*, 1)\not \in \col(R)$ or $\bar s_{(p^*,1)} = 0$.
Since $(\psi, \pi)$ is valid for $(S, \R^n, \R^n)$, it follows that
\[
\sum_{r \in \col(R)}\hat\psi(r)\bar s_r + \sum_{p \in \col(P)} \hat\pi(p)\bar y_p 
= \sum_{(r',0) \in \col(R)}\psi(r')\bar s_r + \sum_{(p',0) \in \col(P)} \pi(p)\bar y_p 
\ge 1.
\]

Now, assume that $(p^*, 1) \in \col(R)$ and $\bar s_{(p^*,1)} \neq 0$.
Since $R\bar s + P \bar y \in S \times \Z_+$ and $\cP \subseteq \R^n \times \{0\}$, we have $\bar s_{(p^*,1)} \in \Z_+$. 
Define $\tilde R \in \R^{n \times (k-1)}$ by its columns $\col(\tilde{R}) := \{ r \in \R^n : (r,0) \in \col(R)\}$, that is, the columns of $\tilde{R}$ are the columns of $R\setminus\{(p^*, 1)\}$ projected to the first $n$ coordinates. 
Similarly, define $\tilde P \in \R^{n \times (\ell+1)}$ by the columns $\col(\tilde{P}) := \{p\in \R^n : (p,0) \in \col(P)\} \cup \{p^*\}$. 

Consider the pair $(\tilde s, \tilde y) \in \R^{k-1} \times \R^{\ell+1}$ defined by $\tilde s_r = \bar s_{(r,0)}$ for each $r\in \col(\tilde{R})$ and $\tilde y_p = \bar y_{(p,0)}$ for each $p\in \col(\tilde{P})\setminus \{p^*\}$ and $\tilde y_{p^*} = \bar y_{(p^*,0)} + \bar s_{(p^*,1)}$. 
Since $\bar s_{(p^*,1)} \in \Z_+$ and $R\bar s + P \bar y \in S\times \Z_+$, it follows that $\tilde R\tilde s + \tilde P \tilde y \in S$. 
Thus, $(\tilde s, \tilde y)\in X_S(\tilde R,\tilde P)$. 
By rearranging terms, we see that
\begin{align*}
& \sum_{r \in \col(R)}\hat\psi(r)\bar s_r + \sum_{p \in \col(P)} \hat\pi(p)\bar y_p\\
=&  \sum_{r \in \col(R)\setminus\{(p^*,1)\}}\hat\psi((r,0))\bar s_r + \hat\psi((p^*,1)) \bar s_{(p^*,1)} +
\sum_{p \in \col(P)} \hat\pi(p)\bar y_p  \\
=& \sum_{r \in \col(R)\setminus\{(p^*,1)\}}\hat\psi((r,0))\bar s_r + V_\psi(p^*) \bar s_{(p^*,1)} +
\sum_{p \in \col(P)} \hat\pi(p)\bar y_p\\
=&\sum_{r \in \col(\tilde{R})}\psi(r)\tilde s_r + \sum_{p \in \col(\tilde{P})} \pi(p)\tilde y_p  \\
\geq& 1,
\end{align*}
where the last inequality holds since $(\psi, \pi)$ is a valid pair for $(S, \R^n, \R^n)$.
\end{cpf}

By Theorem~\ref{thm:extension}, there exist functions $\psi': \R^n \times \R_+\to \R$ and $\pi': \R^n \times \R_+\to \R$ such that $(\psi', \pi')$ is a minimal valid pair for $(S\times \Z_+, \R^n \times \R_+, \R^n \times \R_+)$ and $(\psi', \pi') \le (\hat{\psi}, \hat{\pi})$ on $ \cR \times \cP$.
Thus, by construction of $(\hat{\psi}, \hat{\pi})$, we have $\psi'((r,0)) \le \psi(r)$ and $\pi'((p,0))\le \pi(p)$ for all $r, p \in \R^n$.
Since $\psi$ is a minimal valid function for $(S, \R^n)$, we also have that $\psi'((r,0)) = \psi(r)$ for all $r\in \R^n$.
Similarly, since $\pi$ is a minimal lifting of $\psi$, $\pi'((p,0)) = \pi(p)$ for all $p\in \R^n$.
By definition of $V_{\psi}(p^*)$, this implies that 
\begin{equation}\label{eq:coeff_equal}
\psi'((p^*,1)) = \hat\psi((p^*,1)) = V_{\psi}(p^*).
\end{equation}
We now show $\psi'((r,r_{n+1})) \leq \psi_{\Delta}((r,r_{n+1}))$ for $(r,r_{n+1}) \in \R^n \times \R$. Note that $(r,r_{n+1})=(r,0)+r_{n+1}(p^*,1)$. By Lemma~\ref{lem:additive}, 
\begin{equation}\label{eq:additivity}
\psi_{\Delta}((r, 0))+r_{n+1}\psi_{\Delta}((p^*,1))=\psi_{\Delta}((r,r_{n+1})).
\end{equation} 
Note that $\psi'((r,0)) \le \hat{\psi}((r,0)) = \psi(r) = \psi_{\Delta}((r,0))$.
Thus,
\begin{align}
& ~ \psi'((r,r_{n+1})) \nonumber\\
\leq&~   \psi'((r,0)) + r_{n+1}\psi'((p^*,1)) &\text{by Theorem}~\ref{thm:subadditive}\nonumber\\
\leq&~  \hat\psi((r,0)) + r_{n+1}\hat\psi((p^*,1)) &\text{by}~\eqref{eq:coeff_equal}\nonumber\\
=& ~\psi_{\Delta}((r,0)) + r_{n+1}\psi_{\Delta}((p^*,1))&\label{eq:upper_bound_1}\\
=&~ \psi_{\Delta}((r,r_{n+1})).\nonumber&\text{by}~\eqref{eq:additivity}
\end{align}

Let $p\in \R^n$. 
By Proposition~\ref{prop:easy-partial},~\eqref{eq:upper_bound_1}, and $\pi'((p,0)) = \pi(p)$,  we obtain
\begin{align*}
\pi(p)=  &~ \pi'((p,0))\\
\leq &~ \inf\{\psi'((p,0)+(w,z)): (w,z) \in W^+_{S\times \Z_+}\} \\
\leq &~ \inf\{\psi_{\Delta}((p,0)+(w,z)): (w,z) \in W^+_{S\times \Z_+} \}\\
=&~ \psi^*_{[p^*, B]}((p,0)).
\end{align*} 
\end{proof}

\subsection{Towards a description of the fixing region}

In this subsection, let $B$ be a maximal $S$-free convex $0$-neighborhood of the form~\eqref{eqBDefn}, let $\psi:=\psi_B$ be the valid function for $(S, \R^n)$ obtained from $B$ using~\eqref{eq:psi-from-B}, and let $p^* \in \R^n$.
In this subsection, we define a collection of polyhedra (given by explicit inequalities) whose union $\mathcal{X}(B,p^*)$ will be shown to be a subset of $\mathcal F_{\psi,p^*}$.
The results in this subsection consider the pyramid $\pyr(B, \lambda, p^*)$ only for the value $\lambda = V_{\psi}(p^*)$.
So, in order to reduce notation, we frequently use the notation
\[
\Delta := \pyr(B, V_{\psi}(p^*), p^*).
\]

Let $\tilde B \subseteq \R^d$ be an $S$-free 0-neighborhood that takes one of the following forms: either $\tilde{B} =\{r \in \R^d: a^i\cdot r \leq 1 ~ \forall ~ i \in I\}$ or $\tilde B$ is a pyramid of the form~\eqref{eq:Blambda}, which we write as $\tilde{B} =\{r \in \R^{d-1}\times \R_+: a^i\cdot r \leq 1 ~ \forall ~ i \in I\}$. 
For $x\in \R^d$, the \textit{spindle corresponding to $x$} is defined to be
\begin{equation}
\label{eq:R_B-x}
R_{\tilde B}(x) := \{r\in \R^d: (a^i-a^k)\cdot r\leq 0 \text{ and } (a^i-a^k)\cdot(x- r)\leq 0 ~ \forall i\in I\},
\end{equation}
where $\psi_{\tilde B}$ is defined according to~\eqref{eq:psi-from-B}\footnote{We remind the reader that formula~\eqref{eq:psi-from-B} is well-defined for any choice of $B$ containing $0$.} and $k \in I$ is the index such that $\psi_{\tilde B}(x) = a^k\cdot x$.
Spindles were originally used in~\cite{bcccz,dw2008}.

Theorem~\ref{thm:lifting-fixing} provides a geometric inner approximation of the fixing region $\mathcal F_{\psi,p^*}$. 
The inner approximation is given by the set
\begin{equation}\label{eq:compact_notation}
\mathcal X(B,p^*) :=\left(\bigcup_{(\bar x, \bar x_{n+1}) \in \Delta\cap (S\times \Z_+)} \bigg(\bigcup_{i=0}^{\bar x_{n+1}} R_B( \bar x - \bar x_{n+1}p^*) + ip^*\bigg)\right)
\end{equation}

\begin{theorem}\label{thm:lifting-fixing}
The set $ \cX(B,p^*)$ satisfies $ \cX(B,p^*) + W_S \subseteq  \mathcal{F}_{\psi,p^*}$.
Also, if $\pi \in \mathcal{L}_{\psi,p^*}$, $q \in \cX(B,p^*)$ and $w \in W_S$, then
\[
\pi(q+w) = \pi(q) = \psi^*_{[p^*, B]}((q,0)),
\] 
where $\psi^*_{[p^*,B]}$ is the function defined in~\eqref{eq:ex_trivial_lifting}. 
\end{theorem}


We require some tools to prove Theorem~\ref{thm:lifting-fixing}.
For $q \in \R^n$, consider lifting $q$ \emph{after} $p^*$ has been lifted, that is, consider the smallest value that a minimal lifting of $\psi$ can take at $q$ \emph{after} the lifting is restricted to take value $V_{\psi}(p^*)$ at $p^*$.
To this end, define
\begin{equation}\label{eq:seq_lift}
V_{\psi}(q;~p^*) := \inf\left\{\pi(q):~\pi\in \mathcal{L}_{\psi, p^*}\right\}.
\end{equation}

Proposition~\ref{prop:geometric-V} states that $V_{\psi}(p^*)$ can be computed by constructing the pyramid $\Delta \subseteq \R^n \times \R_+$.
Thus, because $V_{\psi}(q;~p^*)$ is defined \emph{after} $V_{\psi}(p^*)$ is fixed, $\Delta$ should affect $V_{\psi}(q;~p^*)$.
This leads us to examine points $(q,\bar{q}) \in \R^n \times \R_+$.
For $\lambda>0$, $\bar{q} \in \R_+$, and $i \in I$, where $I$ is the index set defining $B$ in~\eqref{eqBDefn}, consider the inequality
\begin{equation}\label{eqFamInequalities}
   a^i\cdot r +(V_{\psi}(p^*) - a^i\cdot p^*) r_{n+1} +
(\lambda - a^i\cdot q - (V_{\psi}(p^*) - a^i\cdot p^*)\bar{q})r_{n+2}   \leq 1.
\end{equation}

We can apply the pyramid operator $\pyr$ defined in~\eqref{eq:Blambda} using $\Delta$ as a base to obtain the iterated pyramid
\begin{equation}\label{eq:B2lambda}
\pyr(\Delta, \lambda, (q , \bar{q})) 
=
 \left\{ (r, r_{n+1}, r_{n+2}) \in  \R^n \times \R_+ \times \R_+ :
\begin{array}{l}
r_{n+1} - \bar{q}r_{n+2} \ge 0 \text{ and }\\
 \eqref{eqFamInequalities} \text{ holds } \forall ~ i \in I
\end{array}
\right\}.
\end{equation}
%
Geometrically, the iterated pyramid in~\eqref{eq:B2lambda} is the pyramid (assuming that it is bounded) in $\R^{n+2}$ with base $\Delta \times \{0\}$ and apex $\frac{1}{\lambda}(q, \bar{q}, 1)$.
%
%
In this new pyramid, the inequality $r_{n+1} - \bar{q}r_{n+2} \ge 0$ is the result of lifting the inequality $r_{n+1}\ge 0$ defining $\Delta$.

The first result that we need to prove Theorem~\ref{thm:lifting-fixing} is the following generalization of a result about spindles in~\cite{bcccz,dw2008}.

\begin{proposition}\label{obs:lifting-region}
Let $(x, x_{n+1}) \in \Delta \cap (S \times \Z_+)$.
If $(q,\bar{q}) \in R_{\Delta}((x, x_{n+1})) \cap (\R^n \times \R_+)$, then
\[
\psi_{\Delta}((q,\bar{q})) = \inf\{\lambda > 0 : 
\pyr(\Delta, \lambda, (q,\bar{q}))  \textrm{ is } (S \times \Z \times \Z)\textrm{-free} 
\},
\]
where $\psi_{\Delta}$ is defined from $\Delta$ using~\eqref{eq:psi-from-B}.
\end{proposition}
The proof of Proposition~\ref{obs:lifting-region} is technical and provided in Appendix~\ref{sec:propObsLiftSpindle}.

The next result shows that $V_{\psi}(q;~p^*)$ can be computed by constructing a pyramid in $\R^{n+2}$ with base $\Delta \times \{0\}$ using~\eqref{eq:B2lambda}. 
So, in order to sequentially lift variables to find a $\pi \in \cL_{\psi, p^*}$, we can repeatedly apply the pyramid operator $\pyr$ using the set from the previous lifted variable as a new base.

\begin{proposition}\label{4d lifting cone}
Let $q \in \R^n$.
The value $V_{\psi}(q; p^*)$ satisfies
\begin{equation}\label{eq:geometric_lifting}
V_{\psi}(q; p^*) = \inf\left\{\lambda >0 :
\pyr(\Delta, \lambda, (q,0)) \textrm{ is } (S\times \Z \times\Z)\text{-}\textrm{free}
\right\}.
\end{equation}
\end{proposition}
The proof of Proposition~\ref{4d lifting cone} is in Appendix~\ref{sec:prop6} and is similar to that of Proposition~\ref{prop:geometric-V}.

Proposition~\ref{4d lifting cone} characterizes $V_{\psi}(q; p^*)$ using a pyramid that depends on $(q,0)$.
Proposition~\ref{prop:invariant V} states that $V_{\psi}(q; p^*)$ can be determined using pyramids that depend on certain translations of $(q,0)$ while holding $p^*$ fixed.

\begin{proposition}\label{prop:invariant V}
If $q \in \R^n$ and $ \bar{q} \in \Z_+$, then
\begin{equation*}
\begin{array}{rl@{\hskip .1 cm}l}
V_{\psi}(q; p^*) = & \inf\{\lambda > 0 : &
\pyr(\Delta, \lambda, (q,0))  \textrm{ is } (S \times \Z \times \Z)\textrm{-free } %
\}\\
= &  
\inf\{\lambda > 0 : &
\pyr(\Delta, \lambda, (q,\bar{q}))  \textrm{ is } (S \times \Z \times \Z)\textrm{-free } 
\}.
\end{array}
\end{equation*}
\end{proposition}

\begin{proof}
The first equation follows from Proposition~\ref{4d lifting cone}. 
Define the linear transformation $U:\R^n\times \R\times \R\to \R^n\times \R\times \R$ by
\begin{equation*}
U\left(y, y_{n+1}, y_{n+2}\right) = \left(y, y_{n+1}+y_{n+2}\bar{q}, y_{n+2}\right).
\end{equation*}
Note that $U$ is invertible and $U^{-1}\left(y, y_{n+1}, y_{n+2}\right) = \left(y, y_{n+1}-y_{n+2}\bar{q}, y_{n+2}\right)$. 
Since $\bar{q}\in \Z$, the map $U$ is unimodular. 
Both $U$ and $U^{-1}$ map $S\times \Z\times\Z$ onto itself, and therefore, they map $(S\times \Z\times\Z)$-free sets to $(S\times \Z\times\Z)$-free sets. 

Let $\lambda >0$ and $(r, r_{n+1}, r_{n+2}) \in \pyr(\Delta, \lambda, (q,0))$. 
For each $i \in I$,~\eqref{eq:B2lambda} implies 
\[
a^i\cdot r +(V_{\psi}(p^*) - a^i\cdot p^*) r_{n+1} + (\lambda - a^i\cdot q)r_{n+2}   \leq 1.
\]
Thus,
\begin{equation}\label{eqSat1}
a^i\cdot r +(V_{\psi}(p^*) - a^i\cdot p^*) (r_{n+1} +r_{n+2}\bar{q}) + (\lambda - a^i\cdot q - (V_{\psi}(p^*) - a^i\cdot p^*)\bar{q})r_{n+2}   \leq 1.
\end{equation}
Equation~\eqref{eqSat1} is equivalent to $U(r, r_{n+1}, r_{n+2}) = (r, r_{n+1} +r_{n+2}\bar{q}, r_{n+2})$ satisfying~\eqref{eqFamInequalities} for each $i \in I$.
Also, $r_{n+1} \ge 0$ because $(r, r_{n+1}, r_{n+2}) \in  \pyr(\Delta, \lambda, (q,0))$. 
Thus, $(r_{n+1} - r_{n+2}\bar{q}) + r_{n+2}\bar{q} = r_{n+1} \ge 0$ and $U(r, r_{n+1}, r_{n+2}) $ satisfies every inequality defining $\pyr(\Delta, \lambda, (q,\bar{q}))$ in~\eqref{eq:B2lambda}.
So, $U(r, r_{n+1}, r_{n+2}) \in \pyr(\Delta, \lambda, (q,\bar{q}))$ and 
\[
U \pyr(\Delta, \lambda, (q,0)) ~ \subseteq ~ \pyr(\Delta, \lambda, (q,\bar{q})).
\]

It remains to show $U^{-1} \pyr(\Delta, \lambda, (q,\bar{q})) \subseteq \pyr(\Delta, \lambda, (q,0))$.
If $(r, r_{n+1}, r_{n+2}) \in \pyr(\Delta, \lambda, (q,\bar{q}))$, then the $(n+1)$-st component of $U^{-1}(r, r_{n+1}, r_{n+2})$ is $r_{n+1} - r_{n+2}\bar{q}$.
Because $(r, r_{n+1}, r_{n+2})$ satisfies the inequalities~\eqref{eq:B2lambda}, we have $r_{n+1} - r_{n+2}\bar{q} \ge 0$.
Using arguments from the first part of this proof, 
we have $U^{-1} \pyr(\Delta, \lambda, (q,\bar{q})) \subseteq \pyr(\Delta,\lambda, (q,0))$.
So, $U \pyr(\Delta, \lambda, (q,0)) = \pyr(\Delta, \lambda, (q,\bar{q}))$. 

Since $U$ and $U^{-1}$ preserve $(S\times \Z \times \Z)$-free sets, the previous argument implies that if $ \pyr(\Delta, \lambda, (q,0))$ is $(S\times \Z \times \Z)$-free, then $ \pyr(\Delta, \lambda, (q,\bar{q}))$ is $(S\times \Z \times \Z)$-free, and vice versa.
This gives the desired result. 
\end{proof}
For $t\in \R$, define 
\(
H_t := \R^n \times \{t\}.
\)
The next proposition shows that translating $H_0 \cap R_{\Delta}(\bar x,\bar x_{n+1})$ by $tp^*$ is equal to projecting $H_t\cap R_{\Delta}(\bar x,\bar x_{n+1})$ onto the first $n$ coordinates.
The proof of Proposition~\ref{prop:projection} is given in Appendix~\ref{AppendixA1}.

\begin{proposition}\label{prop:projection}
If $(\bar x,\bar x_{n+1})\in (S\times \Z_+)\cap \Delta$ is a blocking point of $\Delta$ and $t\in \R$, then
\[
\begin{array}{rcl}
H_t\cap R_{\Delta}(\bar x,\bar x_{n+1}) &=& (H_0 \cap R_{\Delta}(\bar x,\bar x_{n+1}))+t(p^*, 1)\\
& = & (R_B(\bar{x} - \bar{x}_{n+1}p^*) \times \{0\})+t(p^*, 1).
\end{array}
\]
\end{proposition}

We can now prove Theorem~\ref{thm:lifting-fixing}.

\begin{proof}[Proof of Theorem \ref{thm:lifting-fixing}] 
Recall that $\Delta = \pyr(B, V_{\psi}(p^*), p^*)$.
Let $(\bar x, \bar x_{n+1}) \in \Delta \cap (S\times \Z_+)$.
If $\bar x_{n+1} = 0$, then $\bar x\in B \cap S$. 
In this case, 
\(
\bigcup_{i=0}^{\bar x_{n+1}} R_B( \bar x - \bar x_{n+1}p^*) + ip^* = R_B(\bar x).
\)
It is well-known (see, for example, ~\cite{dw2008,dw2,ccz,bcccz}) that $R_B(\bar x) + W_S \subseteq R(B)$, where $R(B)$ is the extended lifting region~\eqref{eq:lifting-region} and $R_B(\bar x)$ is the spindle corresponding to $\bar x$ given in~\eqref{eq:R_B-x}.
Thus, by the definition of $R(B)$, we obtain 
\[
\left(\bigcup_{i=0}^{\bar x_{n+1}} R_B( \bar x - \bar x_{n+1}p^*) + ip^*\right)+W_S = R_B(\bar x) +W_S \subseteq \cF_{\psi, p^*}.
\]

It is left to show
\(
\bigcup_{i=0}^{\bar x_{n+1}} R_B( \bar x - \bar x_{n+1}p^*) + ip^* \subseteq \cF_{\psi, p^*}
\)
when $x_{n+1} \ge 1$, i.e., when $(\bar x, \bar x_{n+1})$ is a blocking point of $\Delta$.
Let $\psi^*_{[p^*, B]} :\R^n \times \R_+ \to \R$ be from~\eqref{eq:ex_trivial_lifting}, 
$\pi\in \cL_{\psi, p^*}$, and
\(
q\in R_B( \bar x - \bar x_{n+1}p^*) + ip^*
\)
for some $i\in \{0, \dots, \bar{x}_{n+1}\}$. 
Note that
\[
\begin{array}{rcll}
V_{\psi}(q; p^*) & \leq& \pi(q) & \text{by the definition of }V_{\psi}(q; p^*)\\
& \le & \psi^*_{[p, B^*]}((q,0)) & \text{by Theorem}~\ref{thm:uni upperbound}\\
& = &\displaystyle \inf_{(w,z) \in W^+_{S\times \Z_+}}\psi_{\Delta}((q, 0)+(w,z) )\\
& \le & \psi_{\Delta}((q,i)).
\end{array}
\]

By Proposition~\ref{prop:projection}, $(q,i)\in R_{\Delta}(\bar{x}, \bar x_{n+1}) \cap H_i$, so by Proposition~\ref{obs:lifting-region} with $\bar{q} = i$,
\[
\psi_{\Delta}((q,i)) = \inf\{\lambda > 0 : 
\pyr(\Delta, \lambda, (q,i))  \textrm{ is } (S \times \Z \times \Z)\textrm{-free} 
\}.
\]
By Proposition~\ref{prop:invariant V}, 
\[
\psi_{\Delta}((q,i)) = V_{\psi}(q; p^*).
\]
Thus, $V_{\psi}(q; p^*)= \pi(q) =  \psi^*_{[p^*, B]}((q,0))$.
Note that $\pi$ was chosen arbitrarily in $\cL_{\psi, p^*}$. 
Hence, every function in $\cL_{\psi, p^*}$ agrees on $q$. 
By definition of $\cF_{\psi, p^*}$ and Proposition~\ref{prop:easy-partial} \emph{(a)}, it follows that
\[
\left(\bigcup_{i=0}^{\bar x_{n+1}} R_B( \bar x - \bar x_{n+1}p^*) + ip^*\right)+W_S  \subseteq \cF_{\psi, p^*}.
\]
\end{proof}

\subsection{Translation invariance of fixing region}\label{subsecTranslation}~

\begin{theorem}\label{thm:fixing_covering_trans}
Let $B$ be a maximal $S$-free convex $0$-neighborhood and let $t\in \R^n$ such that $0\in \intr(B+t)$.
Thus, $B+t$ is a maximal $(S+t)$-free convex $0$-neighborhood. 
For $p^*\in \R^n$ and $\hat{p}:=p^*+V_{\psi}(p^*)t\in \R^n$,
\[
\mathcal X(B,p^*)+W_S=\R^n \quad \text{if and only if} \quad
\mathcal X(B+t,\hat{p})+W_{S+t}=\R^n.
\]
\end{theorem}

Theorem~\ref{thm:fixing_covering_trans} states that if a given maximal $S$-free convex $0$-neighborhood $B$ is one point fixable, then any translation $B+t$ such that $B+t$ is $(S+t)$-free is also one point fixable.
The proof of Theorem~\ref{thm:fixing_covering_trans} is technical in nature and is similar to that of Theorem 3.1 in~\cite{basu-paat-lifting}. For this reason, we provide the proof in Appendix~\ref{sec:translation}. %


\section{Application: Fixing Regions of Type 3 triangles}\label{section:fixing_cover} In this section, we find minimal liftings for Type 3 triangles. 
Type 3 triangles, which are defined precisely below,
are maximal $S$-free convex $0$-neighborhoods in $\R^2$ that contains exactly three  points of $S$, one in the relative interior of each facet. 
In Subsection~\ref{ssec:SuffType3}, we identify conditions that guarantee that a Type 3 triangle is one point fixable. 
In Subsection~\ref{sec:mixing-set-type-3}, we show that a family of Type 3 triangles coming from the extensively studied mixing set problem satisfies this sufficient condition.

In this section, let $S = \Z^2+b$ for $b = (b_1, b_2)\in \R^2\setminus \Z^2$. 
Without loss of generality, we assume that $-1 \leq b_1, b_2 \leq 0$, and, by relabeling the coordinates, we assume that $-1 \leq b_2 \leq b_1 \leq 0$. 
Thus, the origin $(0,0)$ is contained in the interior of the triangle $\conv\{\bar s^1, \bar s^2, \bar s^3\}$, where $\bar s^1 := (1+b_1,1+b_2), \bar s^2 := (b_1,1+b_2),$ and $\bar s^3 := (b_1,b_2).$

For $\gamma_1, \gamma_2,\gamma_3 > 0$ with $\gamma_2, \gamma_3<1$, define the vectors
\begin{align}
q^1= q^1(\gamma_1)&:= \left(\frac{1}{(1, \gamma_1)\cdot(b_1+1, b_2+1)},\frac{\gamma_1}{(1, \gamma_1)\cdot(b_1+1, b_2+1)}\right), \nonumber\\
q^2= q^2(\gamma_2)&:=\left(\frac{-1}{(-1, \gamma_2)\cdot(b_1, b_2+1)},\frac{\gamma_2}{(-1, \gamma_2)\cdot(b_1, b_2+1)}\right), \label{eqTriQvals}\\
q^3= q^3(\gamma_3)&:=\left(\frac{\gamma_3}{(\gamma_3, -1)\cdot(b_1, b_2)}, \frac{-1}{(\gamma_3, -1)\cdot(b_1, b_2)}\right),\nonumber
\end{align}
and the triangle
\begin{equation}\label{eq:type-3}
T(\gamma_1, \gamma_2, \gamma_3) :=\{(x_1, x_2)\in \R^2:~ q^i \cdot (x_1, x_2)\leq 1 ~ \forall ~i\in \{1,2,3\}\}.
\end{equation}

Each triangle in the collection $\{T(\gamma_1, \gamma_2, \gamma_3) : \gamma_1, \gamma_2,\gamma_3 > 0 \text{ and }\gamma_2, \gamma_3<1\}$ is a maximal $S$-free convex $0$-neighborhood in $\R^2$ such that each facet contains one of the points $\overline{s}^1,\overline{s}^2$ and $\overline{s}^3$ from $S$ in their relative interior. 
In the literature, these triangles are referred to as \emph{Type 3 triangles}. See~\cite{dw2008} and the references therein for more on Type 3 triangles and the classification of $S$-free sets in $\R^2$. 

\subsection{Sufficient condition for Type 3 triangles to be one point fixable}\label{ssec:SuffType3}
Let $T := T(\gamma_1, \gamma_2, \gamma_3)$ be a Type 3 triangle.
Using~\eqref{eqTriQvals}, define the pyramid
\begin{align}\label{eq:P}
\begin{array}{l}
P :=\{(x_1, x_2, x_3) \in \R^2 \times \R_+: q^2(\gamma_2) \cdot (x_1, x_2) \leq 1,~\\[.1 cm]
\hspace{1.5 in} q^1(\gamma_1) \cdot (x_1, x_2)+(1-\frac{(1, \gamma_1)\cdot (b_1+1, b_2+2)}{(1, \gamma_1)\cdot(b_1+1, b_2+1)})x_3\leq 1,~\\[.1 cm]
\hspace{1.5 in} q^3(\gamma_3) \cdot (x_1, x_2) +(\frac{1}{2}-\frac{(\gamma_3, -1)\cdot(1+b_1, 2+b_2)}{2(\gamma_3, -1)\cdot(b_1, b_2)})x_3\leq 1\}.
\end{array}
\end{align}

Observe $T\times \{0\} = P\cap \{(x_1, x_2, x_3) :~ x_3 = 0\}$.
Also, $P$ contains the {$S\times \Z$} points $(s^1, z^1) := (1+b_1,1+b_2,0)$, $(s^2, z^2) := (b_1,1+b_2,0)$, $(s^3, z^3) := (b_1,b_2,0)$, $(s^4, z^4) := (1+b_1,2+b_2,1)$, $(s^5, z^5) := (b_1,1+b_2,1)$, and $(s^6, z^6) := (1+b_1,1+b_2,2)$, and $P$ has three facets, $F_1, F_2,$ and $F_3$, containing $\{(s^1, z^1), (s^4, z^4)\}$, $\{(s^2, z^2), (s^5, z^5)\}$, and $\{(s^3, z^3), (s^6, z^6)\}$, respectively. 
In order to apply Theorem~\ref{thm:lifting-fixing} to $T$, we need to show $P = \pyr(T, V_\psi(p^*), p^*)$ for some $p^* \in \R^2$. 
In the next result, we give sufficient conditions on $(\gamma_1, \gamma_2,\gamma_3)$ for such a $p^*$ to exist, i.e., we give sufficient conditions for $T$ to be one point fixable.
Note that $W_S = \Z^2$ because $S = \Z^2 + b$. 

\begin{proposition}\label{prop:covering_prop_type3}
Let $T := T(\gamma_1, \gamma_2, \gamma_3)$ be a Type 3 triangle and $P$ be of the form~\eqref{eq:P}.
Let $\psi$ be the valid function for $(S, \R^2)$ obtained from $T$ using~\eqref{eq:psi-from-B}. 
\begin{enumerate}[leftmargin=*]
\item[(i)] $P$ is a pyramid whose apex $a^* = (a^*_1, a^*_2, a^*_3)$ satisfies $a^*_3>0$ if and only if $\gamma_2(2-\gamma_3+2\gamma_1\gamma_3)-\gamma_1\gamma_3>0$.
\item[(ii)] If $P$ is {$(S\times \Z)$}-free, then $\mathcal{X}(T, p^*)+W_S=\R^n$ for $p^* = \frac{1}{a^*_3} (a^*_1, a^*_2)$.
Thus, $\cL_{\psi, p^*}$ contains a unique function. 
\end{enumerate} \end{proposition}

\begin{proof}
The apex of $P$ is $a^* = (a_1^*, a_2^*, a_3^*)$, where 
\begin{align*}
a_1^* & = b_1+\frac{\gamma_2(2+2\gamma_1-\gamma_3)}{\gamma_2(2-\gamma_3+2\gamma_1\gamma_3)-\gamma_1\gamma_3}, \\ 
a_2^* & = b_2+\frac{\gamma_1(2-\gamma_3+2\gamma_2\gamma_3)-(1+\gamma_2)(-2+\gamma_3)}{\gamma_2(2-\gamma_3+2\gamma_1\gamma_3)-\gamma_1\gamma_3}, \\
a_3^* & = \frac{2(1+\gamma_1+\gamma_2-\gamma_2\gamma_3)}{\gamma_2(2-\gamma_3+2\gamma_1\gamma_3)-\gamma_1\gamma_3}.
\end{align*}

In order for $P$ to be a {pyramid} with apex $a^*$ satisfying $a^*_3>0$, it is enough to show that $2(1+\gamma_1+\gamma_2-\gamma_2\gamma_3)>0$ and $\gamma_2(2-\gamma_3+2\gamma_1\gamma_3)-\gamma_1\gamma_3>0$. 
The first inequality holds since $\gamma_3<1$ and the second holds by hypothesis. 
Hence, (i) holds. 

By Proposition \ref{prop:fill-in}, $\cL_{\psi, p^*}$ is nonempty. 
By Theorem~\ref{thm:lifting-fixing}, in order to see that $\cL_{\psi, p^*}$ contains a unique function, it is sufficient to show that $\mathcal X(T, p^*)+W_S =  \mathcal X(T, p^*)+\Z^2 = \R^2$. We draw inspiration from~\cite{dw2008}. The crucial observation is that $P = \pyr(T, V_{\psi}(p^*), p^*)$ for the choice of $p^*$ in the hypothesis. 

Figure 8 in~\cite{dw2008} labels the vertices of the spindles $R_T(s^1)$, $R_T(s^2)$ and $R_T(s^3)$ for $T$ (recall~\eqref{eq:R_B-x} for the definition of a spindle). 
For completeness, we reproduce the labels in Figure~\ref{fig:holes} with the values $v^i$ and $\delta_i$ defined below.
The vertices of $T$ are
\begin{align*}
v^1&=\left(b_1+\frac{1+\gamma_1}{1+\gamma_1\gamma_3}, b_2+\frac{\gamma_3+\gamma_1\gamma_3}{1+\gamma_1\gamma_3}\right),\\
v^2&=\left(b_1+\frac{\gamma_2}{\gamma_1+\gamma_2}, b_2+\frac{1+\gamma_1+\gamma_2}{\gamma_1+\gamma_2}\right), \text{ and}\\
v^3&=\left(b_1+\frac{-\gamma_2}{1-\gamma_2\gamma_3}, b_2+\frac{-\gamma_2\gamma_3}{1-\gamma_2\gamma_3}\right).
\end{align*}
The values of $\delta_i$ for $i\in\{1,2,3\}$ are
\begin{equation*}
\delta_1= \frac{1+\gamma_1\gamma_3}{1+\gamma_1+\gamma_2-\gamma_2\gamma_3},~~~~\delta_2=\frac{\gamma_1+\gamma_2}{1+\gamma_1+\gamma_2-\gamma_2\gamma_3},~~~~
\delta_3=\frac{1+\gamma_1-\gamma_2\gamma_3-\gamma_1\gamma_2\gamma_3}{1+\gamma_1+\gamma_2-\gamma_2\gamma_3}.
\end{equation*}
The $\delta_i$'s are convex coefficients satisfying $s^i = \delta_iv^i + (1-\delta_i)v^{i+1}$ for $i=1,2, 3$, where $v^4 = v^1$. One observes that $\delta_i \in [0,1]$ holds because $\gamma_i > 0$ and $\gamma_2, \gamma_3 < 1$. 

\begin{figure}
\begin{tabular}{rl}
\begin{minipage}{.6\textwidth}
\begin{tikzpicture}[scale = 1.5]
\draw[black,fill=white] (0,0) circle (.2ex);
\node[]at(.15, -.15)  {\small $s^3$};
\draw[black,fill=white] (0,2) circle (.2ex);
\node[]at(-.2, 2)  {\small $s^2$};
\draw[black,fill=white] (2,2) circle (.2ex);
\node[]at(2.2, 2)  {\small $s^1$};

\draw[fill = black!30, opacity = .5] (0.5, 2.5)  -- (1.25, 3.25) -- (2, 3.50025) -- (2, 2.50025) -- (1.25, 1.75025) -- (0.5, 1.5) -- cycle;

\draw[black,fill=black] (.5,1) circle (.2ex);
\node[]at(.35, 1) {\small $o$};
\draw[black,fill=black] (2., 1.4995) circle (.2ex);
\node[]at(2.2, 1.4995)  {\small $c^1$};
\draw[black,fill=black] (0.5, 1.5) circle (.2ex);
\node[]at(.35, 1.5)  {\small $e^1$};
\draw[black,fill=black] (0.5, 2.5) circle (.2ex);
\node[]at(.35, 2.6)  {\small $c^2$};
\draw[black,fill=black] (0., 0.5) circle (.2ex);
\node[]at(-.2, 0.5)  {\small $e^2$};
\draw[black,fill=black] (-0.75, -0.25) circle (.2ex);
\node[]at(-.9, -.25)  {\small $c^3$};
\draw[black,fill=black] (1.25, 1.24975) circle (.2ex);
\node[]at(1.4, 1.1)  {\small $e^3$};

\draw[draw = black](.5,1) -- (2., 1.4995)--(2,2) -- (0.5, 1.5) -- cycle;
\draw[draw = black](.5,1) -- (0.5, 2.5) -- (0,2) -- (0., 0.5) -- cycle;
\draw[draw = black](.5,1) --  (-0.75, -0.25) --(0,0) -- (1.25, 1.24975) -- cycle;

\draw[black,fill=black] (1.25, 1.75025) circle (.2ex);
\node[]at(1.25, 1.6)  {\small $g$};
\draw[black,fill=black] (2, 2.50025) circle (.2ex);
\node[]at (2.15, 2.50025) {\small $i$};
\draw[black,fill=black] (2, 3.50025) circle (.2ex);
\node[]at(2.15, 3.50025)  {\small $j$};
\draw[black,fill=black] (2, 3.00025) circle (.2ex);
\node[]at(2.15, 3.00025)  {\small $m$};
\draw[black,fill=black] (1.25, 3.25) circle (.2ex);
\node[]at(1.1, 3.25)  {\small $k$};
\draw[black,fill=black] (0.5, 2) circle (.2ex);
\node[]at(.35, 2)  {\small $l$};
\draw[black,fill=black] (1.25, 2.25025) circle (.2ex);
\node[]at(1.45, 2.15)  {\small $u$};
\draw[black,fill=black] (1.25, 2.75) circle (.2ex);
\node[]at(1.1,2.75)  {\small $p^*$};

\draw[dashed] (0.5, 2) -- (1.25, 2.75) -- (2, 3.00025) -- (1.25, 2.25025) -- cycle;
\draw[dashed] (1.25, 2.75) --(1.25, 3.25);
\draw[dashed] (1.25, 2.25025)--(1.25, 1.75025);

\draw[dotted] (.5, 3) -- (-1.5, -1) -- (2.5, 1.666)--cycle;

\draw[black,fill=black]  (.5, 3) circle (.2ex) node[above]{\small $v^2$};
\draw[black,fill=black] (2.5, 1.666) circle (.2ex) node[right]{\small $v^1$};
\draw[black,fill=black] (-1.5, -1) circle (.2ex) node[above left]{\small $v^3$};
\end{tikzpicture}
\end{minipage}
&
\begin{minipage}{.3\textwidth}
$\begin{array}{rl}
c^1 &= \delta_1 v^1 \\
c^2 &= \delta_2 v^2 \\
c^3 & = \delta_3 v^3 \\
e^1 & = (1-\delta_1)v^2\\
e^2 & = (1-\delta_2)v^3\\
e^3 & = (1-\delta_3)v^1\\
 g &= s^1-e^3\\
  i &= g-c^3+e^2\\
j &=i-e^1+c^2\\
m &=\frac{1}{2}(i+j)\\
k & = j-g+e^1\\
 l & = \frac{1}{2}(e^1+c^2)\\
 p^* & = k-m+i \\
 u &=g-i+m\\
\end{array}$
\end{minipage}
\end{tabular}
\caption{The spindles of $T$ given in~\cite{dw2008}. $K := \conv\{c^2, k, j, i, g, e^1\}$ is shaded, and $o$ is the origin.\label{fig:holes}}
\end{figure}
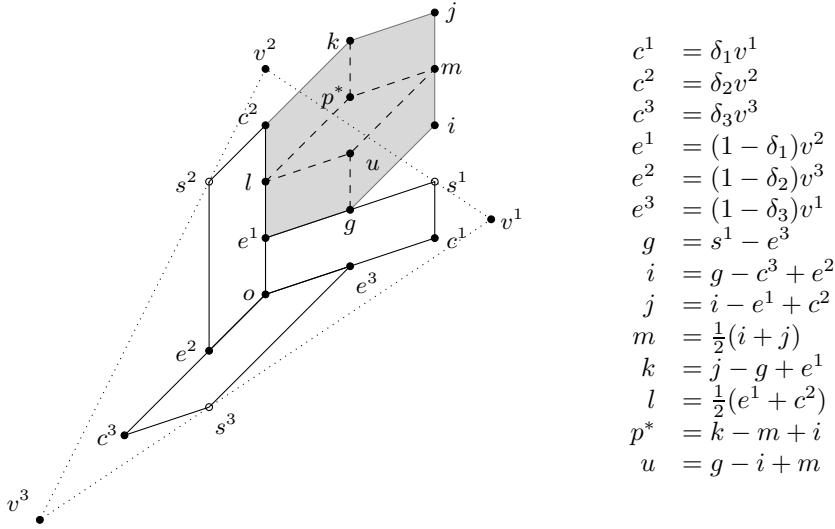

Define the region $K := \conv\{c^2, k, j, i, g, e^1\}$ (see Figure~\ref{fig:holes}). The literature~\cite{dw2008,ccz},~\cite[Theorem 4]{bcccz} shows that $\R^2\setminus (K + \Z^2)$ is contained in $R(T)$, which is contained in $\mathcal{X}(T,p^*)+\Z^2$.
Hence, if we can show that $K\subseteq \mathcal{X}(T,p^*)+\Z^2$, then $K + \Z^2\subseteq \mathcal{X}(T,p^*)+\Z^2$ implying that $\R^2 = \mathcal{X}(T, p^*) + \Z^2$, thus completing the proof. 

To this end, write $K$ as $K = \cup_{i=1}^ 5 K_i$, where
\[
\begin{array}{lll}
K_1 = \conv\{l, e^1, g, u\}, &
K_2 = \conv\{u, m, i, g\},\\
K_3  = \conv\{m,j,k,p^*\}, &
K_4 = \conv\{c^2, k, p^*, l\}, &
 \text{and } K_5 = \conv\{l, p^*, m, u\}.
\end{array}
\]

\begin{CLAIM}\label{eq:H_i-cases}
$K_1 \subseteq R_T(s^4-p^*)$, $K_2\subseteq R_T(s^5-p^*)+(1,1)$, $K_3 \subseteq R_T(s^4-p^*)+p^*$, $K_4 \subseteq R_T(s^5-p^*)+p^*$, and $K_5\subseteq R_T(s^6-2p^*)+p^*$. 
\end{CLAIM} 

The proof of Claim~\ref{eq:H_i-cases} is technical and appears in Appendix~\ref{appendix:case_analysis}.
By Theorem~\ref{thm:lifting-fixing}, $R_T(s^4-p^*), R_T(s^4-p^*)+p^*, R_T(s^5-p^*)+(1,1), R_T(s^5-p^*)+p^*,$ and $R_T(s^6-2p^*)+p^*$ are contained in $\mathcal{X}(T, p^*)+\Z^2$.
Thus, by Claim~\ref{eq:H_i-cases}, $K \subseteq \mathcal{X}(T, p^*)+\Z^2$. 
\end{proof}

\subsection{Type 3 triangles from the mixing set}\label{sec:mixing-set-type-3}
Proposition~\ref{prop:covering_prop_type3} assumes the pyramid $P$ is $(S\times\Z)$-free. This is satisfied by {\em mixing set} Type 3 triangles~\cite{dw2,mixingSet}. 

The mixing set is considered a fundamental set in mixed-integer programming theory. 
The facet-defining inequalities of this set are called ``mixing'' inequalities as they are supposed to ``mix'' the well-known mixed-integer rounding (MIR) inequalities.
The mixing set appears as a relaxation of several important problems~\cite{mixingSet} such as production planning, capacitated facility location, and capacitated network design. 
Recently, inequalities closely related to mixing inequalities have had a huge impact in solving stochastic integer programs~\cite{luedtke2010integer}. 
Mixing inequalities can be used for general mixed-integer linear programs, and there are several studies of its properties~\cite{dey2011note,dash2009mixing}. 
Several generalizations of the mixing set have been studied as well~\cite{van2005continuous,conforti2007mixed,conforti2007mixing}.

If a Type 3 triangle satisfies 
\(
b\in \intr(\conv\left\{(0,-1), (0,-1/2), (-1,-1)\right\}),
\)
then we say that it is a \emph{mixing set Type 3 triangle}.
With this additional constraint on $b$, the mixing set Type 3 triangles are defined by $b$ satisfying $-1 < b_2 < b_1 < 0$ and $b_1 - 2b_2 > 1$. 
Define $\delta_b = -b_1^2-b_2^2+b_1b_2 - b_2$ and observe that 
\[
\delta_b := b_1(b_2-b_1)-b_2(1+b_2)>0.
\] 

Consider the Type 3 triangle  $T(b) := T(\frac{b_2-b_1}{b_1}, \frac{b_1-b_2}{1+b_1}, \frac{b_1}{b_1-b_2-1})$ defined by
\[
\begin{array}{r}
T(b) = \{(x_1, x_2)\in \R^2: ~ (\frac{-b_1}{\delta_b})x_1+(\frac{b_1-b_2}{\delta_b})x_2\leq 1, ~ (\frac{-b_1-1}{\delta_b})x_1+(\frac{b_1-b_2}{\delta_b})x_2\leq 1,~ \\
  ~ ( \frac{-b_1}{\delta_b})x_1+(\frac{b_1-b_2-1}{\delta_b})x_2\leq 1\}.
\end{array}
\]

By construction, $T(b)\cap S = \{(b_1,b_2), (b_1,1+b_2), (1+b_1,1+b_2)\}$. 
Note that the constraints on $b$ imply that $\gamma_1, \gamma_2, \gamma_3 > 0$ and $\gamma_2, \gamma_3 < 1$, as required.
Substituting these values of $\gamma_1, \gamma_2, \gamma_3$ into~\eqref{eq:P}, we obtain the pyramid
\begin{equation}
\begin{array}{r}
P(b) := 
\{(x_1, x_2, x_3)\in \R^2 \times \R_+: (\frac{-b_1}{\delta_b})x_1+(\frac{b_1-b_2}{\delta_b})x_2-(\frac{b_1-b_2}{\delta_b})x_3\leq 1,~\\[.1 cm]
(\frac{-b_1-1}{\delta_b})x_1+(\frac{b_1-b_2}{\delta_b})x_2\leq 1,\label{eq:defn_p_b}~\\[.1 cm]
 ( \frac{-b_1}{\delta_b})x_1+(\frac{b_1-b_2-1}{\delta_b})x_2+(\frac{2-b_1+2b_2}{2\delta_b})x_3\leq 1\}.
 \end{array}
\end{equation}

We verify the two conditions in Proposition~\ref{prop:covering_prop_type3} to conclude that there exists a $p^*\in \R^2$ satisfying one point fixability for mixing set triangles. 
The condition $\gamma_2(2-\gamma_3+2\gamma_1\gamma_3)-\gamma_1\gamma_3>0$ can be checked using $\gamma_1= \frac{b_2-b_1}{b_1}, \gamma_2 = \frac{b_1-b_2}{1+b_1}, \gamma_3 = \frac{b_1}{b_1-b_2-1}$, and the constraints $-1 < b_2 < b_1 < 0$. 
Next, we verify $\intr(P(b))\cap {(S\times \Z)} =\emptyset$.

\begin{proposition}\label{prop: mix_pyramid_lattice_free}
$P(b)$ is {$(S\times \Z)$}-free if $T(b)$ is a mixing set Type 3 triangle. 
\end{proposition}
\begin{proof} 
For $t\in \Z_+$, define $H_t:= \R^2 \times \{t\}$. 
Since $P(b)\cap H_0 = T(b)\times \{0\}$ is $S$-free, we only need to show $\relintr(P(b)\cap H_t)\cap (S \times \{t\})=\emptyset$ for $t\geq 1$. 

For $t\geq 1$, define the split sets
\begin{align*}
C_1&:= \{(x_1,x_2,t)\in \R^3 :~ t\leq x_2\leq t+1\} + (b_1,b_2,0),\\
C_2&:= \{(x_1,x_2,t)\in \R^3 :~ 0\leq -2x_1+x_2\leq 1\} + (b_1,b_2,0), \text{ and}\\
C_3&:= \bigg\{(x_1, x_2,t)\in \R^3:~ \frac{t}{2}\leq -x_1+x_2\leq \frac{t}{2}+\frac{1}{2}\bigg\} + (b_1,b_2,0).
\end{align*}

The splits $C_1, C_2$ and $C_3$ have no points from $S\times \{t\}$ in their relative interior. 
So, if we show $\relintr(P(b)\cap H_t) \subseteq \relintr(C_1)\cup \relintr(C_2)\cup \relintr(C_3)$, then $P(b)$ will be {$(S\times \Z)$}-free, completing the proof.
To this end, assume $(x_1^*, x_2^*, t)\in \relintr(P(b)\cap H_t) \setminus (\relintr(C_1)\cup \relintr(C_2))$. This implies that $(x_1^*, x_2^*,t)$ does not strictly satisfy some inequality defining $C_1$ and some inequality defining $C_2$. This leads to four cases.

\smallskip

\noindent \textit{Case 1.} Suppose  $x_2^*-b_2\leq t$ and $-2(x_1^*-b_1)+(x_2^*-b_2)\leq 0$. Observe that
\begin{equation*}\begin{array}{rll}
&& (\frac{-b_1}{\delta_b})x^*_1+(\frac{b_1-b_2-1}{\delta_b})x^*_2+(\frac{2-b_1+2b_2}{2\delta_b}) \\
& \geq &  (\frac{-b_1}{\delta_b})(\frac{2b_1 + x^*_2 - b_2}{2})+(\frac{b_1-b_2-1}{\delta_b})x^*_2+(\frac{2-b_1+2b_2}{2\delta_b}) \\
& = & (\frac{b_1-2b_2-2}{2\delta_b})x^*_2 + (\frac{2-b_1+2b_2}{2\delta_b})t + (\frac{-2b^2_1 + b_1b_2}{2\delta_b}) \\
 &\geq &  (\frac{b_1-2b_2-2}{2\delta_b})(t+b_2) + (\frac{2-b_1+2b_2}{2\delta_b})t + (\frac{-2b^2_1 + b_1b_2}{2\delta_b}) = 1.
\end{array}\end{equation*}

The first inequality follows from the assumption $-2(x_1^*-b_1)+(x_2^*-b_2)\leq 0$, and the second from the assumption $x_2^*-b_2\leq t$. This contradicts $(x_1^*, x_2^*, t)\in \relintr(P(b)\cap H_t)$ because the third inequality defining $P(b)$ is violated.

\smallskip

\noindent \textit{Case 2.} Suppose $x_2^*-b_2 \leq t$ and $-2(x_1^*-b_1)+(x_2^*-b_2)\geq 1$. We claim $(x_1^*, x_2^*, t)\in \relintr(C_3)$. It is sufficient to show $\frac{t}{2}< -(x_1^*-b_1)+(x_2^*-b_2) < \frac{t}{2}+\frac{1}{2}$. Because $(x_1^*, x_2^*, t)\in \relintr(P(b)\cap H_t)$, the third inequality defining $P(b)$ bounds $x_2^*$: \begin{equation*}
x_2^*> \frac{-b_1}{1+b_2-b_1}x^*_1+\frac{t}{2}+\frac{1+b_2}{2(1+b_2-b_1)}t+\frac{-\delta_b}{1+b_2-b_1}.
\end{equation*}
Using this, we see that
\begin{equation*}
\begin{array}{rl}
&-(x_1^*-b_1)+(x_2^*-b_2)\\[.1 cm]
> &~ -(x_1^*-b_1) +(\frac{-b_1}{1+b_2-b_1}x^*_1+\frac{t}{2}+\frac{1+b_2}{2(1+b_2-b_1)}t+\frac{-\delta_b}{1+b_2-b_1})-b_2\\[.1 cm]
=&~  \frac{t}{2} +(\frac{-1-b_2}{1+b_2-b_1})x_1^*+(\frac{1+b_2}{2(1+b_2-b_1)})t+(\frac{b_1+b_1b_2}{1+b_2-b_1})\\[.1 cm]
\geq&~  \frac{t}{2} +(\frac{-1-b_2}{1+b_2-b_1})(\frac{x^*_2-b_2-1+2b_1}{2})+(\frac{1+b_2}{2(1+b_2-b_1)})t+(\frac{b_1+b_1b_2}{1+b_2-b_1})\\
=&~  \frac{t}{2} +(\frac{-1-b_2}{2(1+b_2-b_1)})x_2^*+(\frac{1+b_2}{2(1+b_2-b_1)})t+(\frac{2b_2+b_2^2+1}{2(1+b_2-b_1)})\\[.1 cm]
\geq&~   \frac{t}{2} +(\frac{-1-b_2}{2(1+b_2-b_1)})(t+b_2)+(\frac{1+b_2}{2(1+b_2-b_1)})t+(\frac{2b_2+b_2^2+1}{2(1+b_2-b_1)})\\[.1 cm]
=&~ \frac{t}{2}+ \frac{1+b_2}{2(1+b_2-b_1)}>~ \frac{t}{2}.
\end{array}
\end{equation*}
The second inequality follows from $-2(x_1^*-b_1)+(x_2^*-b_2)\geq 1$ and $\frac{-1-b_2}{-b_1+b_2+1}<0$, the third follows from $x^*_2\leq t+b_2$, and the fourth follows from $\frac{1+b_2}{2(1+b_2-b_1)}>0$. 

Since $(x_1^*, x_2^*, t)\in \relintr(P(b)\cap H_t)$, the second inequality defining $P(b)$ implies
\begin{equation*}\begin{array}{rll}
 -(x_1^*-b_1)+(x_2^*-b_2)
&< & -x_1^*+b_1+(\frac{\delta_b}{b_1-b_2}+\frac{1+b_1}{b_1-b_2}x_1^*)-b_2\\
&= & (\frac{1+b_2}{b_1-b_2})x_1^*+\frac{-b_2-b_1b_2}{b_1-b_2}\\
&\leq & (\frac{1+b_2}{b_1-b_2})(\frac{2b_1+x_2^*-b_2-1}{2})+\frac{-b_2-b_1b_2}{b_1-b_2}\\
 &= & (\frac{1+b_2}{2(b_1-b_2)})x_2^*+(\frac{2b_1-4b_2-b_2^2-1}{2(b_1-b_2)})\\
&\leq & (\frac{1+b_2}{2(b_1-b_2)})(t+b_2)+(\frac{2b_1-4b_2-b_2^2-1}{2(b_1-b_2)})\\
&= & \frac{t}{2}+(\frac{1-b_1+2b_2}{b_1-b_2})\frac{t}{2}+(\frac{2b_1-3b_2-1}{2(b_1-b_2)})\\
&\leq & \frac{t}{2}+(\frac{1-b_1+2b_2}{b_1-b_2})\frac{1}{2}+(\frac{2b_1-3b_2-1}{2(b_1-b_2)}) = \frac{t}{2}+\frac{1}{2}.
\end{array}\end{equation*}
The second inequality follows since $-2(x_1^*-b_1)+(x_2^*-b_2)\geq 1$ and $\frac{1+b_2}{-b_1+b_2+1}>0$, the third follows since $x_2^*\leq t+b_2$, and the fourth follows since $t\geq 1$ and $1<b_1-2b_2$. 

\noindent \textit{Case 3.} Suppose $x_2^*-b_2\geq t+1$ and $-2(x_1^*-b_1)+(x_2^*-b_2)\leq 0$. 
Observe that
\begin{equation*}\begin{array}{rll}
&&(\frac{-b_1}{\delta_b})x_1^*+(\frac{b_1-b_2}{\delta_b})x_2^*-(\frac{b_1-b_2}{\delta_b})t\\[.1 cm]
&\geq & (\frac{-b_1}{\delta_b})(\frac{2b_1+x_2^*-b_2}{2})+(\frac{b_1-b_2}{\delta_b})x_2^*-(\frac{b_1-b_2}{\delta_b})t\\[.1 cm]
&=& (\frac{b_1-2b_2}{2\delta_b})x_2^*-(\frac{b_1-b_2}{\delta_b})t+(\frac{-2b_1^2+b_1b_2}{2\delta_b})\\[.1 cm]
&\geq &  (\frac{b_1-2b_2}{2\delta_b})(t+1+b_2)-(\frac{b_1-b_2}{\delta_b})t+(\frac{-2b_1^2+b_1b_2}{2\delta_b})\\[.1 cm]
&= & (\frac{-b_1}{2\delta_b})t +(\frac{b_1}{2\delta_b})+1\\[.1 cm]
&\geq &  (\frac{-b_1}{2\delta_b}) +(\frac{b_1}{2\delta_b})+1 = 1.
\end{array}\end{equation*}
 The first inequality follows since $\frac{-b_1}{\delta_b}>0$ and $-2(x_1^*-b_1)+(x_2^*-b_2) \geq 0$, the second inequality follows since $b_1-2b_2>1$ and $x_2^*\geq t+1+b_2$, and the third inequality follows since $t\geq 1$. This contradicts that $(x_1^*, x_2^*, t)\in \relintr(P(b)\cap H_t)$ because the first inequality defining $P(b)$ is violated.

\smallskip

\noindent \textit{Case 4.} Suppose $x_2^*-b_2\geq t+1$ and $-2(x_1^*-b_1)+(x_2^*-b_2)\geq 1$. 
Observe
\begin{equation*}\begin{array}{rll}
(\frac{-b_1-1}{\delta_b})x_1^*+(\frac{b_1-b_2}{\delta_b})x_2^*  &\geq & (\frac{-b_1-1}{\delta_b})(\frac{x_2^*-1+2b_1-b_2}{2}) + (\frac{b_1-b_2}{\delta_b})x_2^* \\[.1 cm]
 &= & (\frac{b_1-2b_2-1}{2\delta_b})x_2^*+ (\frac{-b_1-1}{\delta_b})(\frac{2b_1-b_2-1}{2}) \\[.1 cm]
&\geq & (\frac{b_1-2b_2-1}{2\delta_b})(2+b_2)+(\frac{-b_1-1}{\delta_b})(\frac{2b_1-b_2-1}{2}) \\[.1 cm]
 &= & 1+\frac{b_1-2b_2-1}{2\delta_b} > 1.
\end{array}\end{equation*}
The first inequality comes from $\frac{-b_1-1}{\delta_b}<0$ and $-2(x_1^*-b_1)+(x_2^*-b_2)\geq 1$. 
The second inequality comes from the fact that $b_1-2b_2>1$ and $\delta_b>0$ so the term $\frac{b_1-2b_2-1}{2\delta_b}$ is positive; furthermore, $x_2^* \geq t+1+b_2\geq 2+b_2 >0$ since $t\geq 1$ and $-1<b_2$. 
The last inequality follows because $\delta_b>0$ and $b_1-2b_2>1$. This contradicts $(x_1^*, x_2^*, t)\in \relintr(P(b)\cap H_t)$ as the second inequality defining $P(b)$ is violated.
\end{proof}

\section{Acknowledgements}

The authors would like to thank the anonymous referees for their extremely helpful comments. Their suggestions helped us correct a flaw in one of the major results and make the material more presentable. 


\bibliographystyle{siam}
\bibliography{full-bib}


\appendix

\section{Nonemptiness of $\mathcal{L}_{\psi,p^*}$}\label{appendix:useful-obs}
\begin{proposition}\label{prop:fill-in}
$\mathcal{L}_{\psi,p^*}$ is nonempty.
\end{proposition}
\begin{proof} 
Define 
\begin{equation*}\label{eq:phi_defn}
\phi(p) := \inf_{w \in \R^n, N \in \N}\bigg\{\psi(w) + NV_\psi(p^*): w + Np^* \in p + W_S\bigg\}.
\end{equation*}
It was shown in~\cite{dw2008} that $\phi$ is a lifting of $\psi$ and $\phi(p^*) = V_\psi(p^*)$. The proof in \cite{dw2008} considers $\R^2$ and $S = \Z^2$, but the proof holds for $\R^n$ and general $S \subseteq \R^n$.
By Zorn's Lemma, there is a minimal lifting $\pi$ of $\psi$ such that $\pi \le \phi$.
By~\eqref{eqVpsip}, $ V_\psi(p^*) \le \pi(p^*) \le \phi(p^*) =  V_\psi(p^*)$. 
Thus, $\pi\in\mathcal L_{\psi,p^*}$ by~\eqref{def:L-psi}.
\end{proof}

\section{Proof of Proposition~\ref{obs:lifting-region}}\label{sec:propObsLiftSpindle}

\begin{proof}[Proof of Proposition~\ref{obs:lifting-region}]~
Recall $\Delta := \pyr(B, V_{\psi}(p^*), p^*).$
Define $\bar{\lambda} := \psi_{\Delta}((q,\bar{q})) $.
Since $\psi_{\Delta}$ takes the form~\eqref{eq:psi-from-B}, there exists some $k \in I$ such that 
\begin{equation}\label{eqSpindleForQ}
\bar{\lambda} = a^{k}\cdot q + (V_{\psi}(p^*) - a^{k}\cdot p^*)\bar{q} \ge a^{i}\cdot q + (V_{\psi}(p^*) - a^{i}\cdot p^*)\bar{q} ~~\forall i \in I.
\end{equation}
As $(q, \bar{q}) \in R_{\Delta}(x, x_{n+1})$, it follows that $\psi_{\Delta}((x, x_{n+1})) = a^{k}\cdot x + (V_{\psi}(p^*) - a^{k}\cdot p^*)x_{n+1}$.

We claim $(x, x_{n+1},1)  \in \pyr(\Delta, \bar{\lambda}, (q,\bar{q}))$.
It suffices to check that $(x, x_{n+1},1) $ satisfies the inequalities from~\eqref{eqFamInequalities} hold.
Let $i \in I$ and consider the corresponding equation given in~\eqref{eqFamInequalities}.
By~\eqref{eqSpindleForQ} and $(x, x_{n+1}) \in \Delta$, 
\begin{align*}
& ~ a^i\cdot x + (V_{\psi}(p^*)- a^i\cdot p^*)x_{n+1} + (\bar{\lambda} - a^i\cdot q - (V_{\psi}(p^*) - a^i\cdot p^*)\bar{q})\\
\le & ~a^i\cdot x + (V_{\psi}(p^*)- a^i\cdot p^*)x_{n+1}  ~ \le  1.
\end{align*}
Thus, $(x, x_{n+1},1)  \in \pyr(\Delta, \bar{\lambda}, (q,\bar{q}))$.
The latter inequality is an equation when $i = k$, which implies that $(x, x_{n+1},1) \in \intr(\pyr(\Delta, \bar{\lambda}-\epsilon, (q,\bar{q}))$ for $\epsilon > 0$ such that $\bar{\lambda}-\epsilon > 0$.
To complete the proof, it is enough to show $\pyr(\Delta, \bar{\lambda}, (q,\bar{q}))$ is $(S\times \Z \times \Z)$-free.

Let $(s,z_{n+1}, z_{n+2}) \in S\times \Z \times \Z$; we want $(s, z_{n+1}, z_{n+2}) \not \in \intr (\pyr(\Delta, \bar{\lambda}, (q,\bar{q}))$. 
By definition of $\pyr(\Delta, \bar{\lambda}, (q,\bar{q}))$, we assume $z_{n+2} \ge 0$. 
If $z_{n+1} < 0$, then $z_{n+1} - \bar{q}z_{n+2} < 0$ because $\bar{q}, z_{n+2} \ge 0$. So, the inequality $r_{n+1} - \bar{q}r_{n+2} \ge 0$ in~\eqref{eq:B2lambda} separates $(s,z_{n+1}, z_{n+2}) $ from $\intr(\pyr(\Delta,\bar{\lambda}, (q,\bar{q}))$.  

Assume $z_{n+1} \ge 0$.
Since $(s, z_{n+1}) \in S\times \Z_+$ and $\Delta$ is $(S\times \Z)$-free, there exists an $i \in I$ such that $a^i\cdot s + (V_{\psi}(p^*) - a^i \cdot p^*) z_{n+1} \ge 1$.
Thus,
\begin{align*}
& ~ a^i\cdot s + (V_{\psi}(p^*) - a^i \cdot p^*) z_{n+1} + (\bar{\lambda} - a^i\cdot q - (V_{\psi}(p^*) - a^i\cdot p^*)\bar{q}) z_{n+2}\\
\ge & ~ 1 +  (\bar{\lambda} - a^i\cdot q - (V_{\psi}(p^*) - a^i\cdot p^*)\bar{q}) z_{n+2}\\
\ge & ~1,
\end{align*}
where the last inequality follows from~\eqref{eqSpindleForQ}.
This completes the proof.
\end{proof}

\section{Proof of Proposition~\ref{4d lifting cone}}\label{sec:prop6}

\begin{proof}[Proof of Proposition~\ref{4d lifting cone}] 
Recall $\Delta := \pyr(B, V_{\psi}(p^*), p^*).$
Consider the model
\begin{equation}\label{eq:int_model}
\bigg\{(s, y_{p^*}, y_q)\in \R^n_+\times \Z_+\times \Z_+: \sum_{r\in \R^n}rs_r+p^*y_{p^*}+qy_{q}\in S\bigg\}.
\end{equation}
Note that $(s, y_{p^*}, y_q)\in \eqref{eq:int_model}$ if and only if $(s, y_{p^*}, y_q) \in \R^n_+\times \R_+\times \R_+$ and
{
\begin{equation}\label{eq:cts_model}
\sum_{r\in \R^n}(r,0,0) s_r+
(p^*,1,0) y_{p^*}+
(q,0,1)y_{q}\in S\times\Z\times\Z.
\end{equation}
}
\begin{CLAIM}
\label{claim:lift_lambda_once} Let $\lambda>0$. If the inequality
\begin{equation}\label{eq:lambda_lift}
\sum_{r\in \R^n}\psi(r)s_r+V_\psi(p^*)y_{p^*}+\lambda y_{q}\geq 1
\end{equation}
is valid for \eqref{eq:int_model}, then $\pyr(\Delta, \lambda, (q,0))$ is $(S\times\Z\times\Z)$-free.
\end{CLAIM}
\begin{cpf}
Let $(\overline{x}, \overline{x}_{n+1}, \overline{x}_{n+2})\in S\times \Z\times\Z$. 
If $ \overline{x}_{n+1} < 0$ or $ \overline{x}_{n+2}<0$, then by the definition of $\pyr(\Delta, \lambda, (q,0))$, $(\overline{x}, \overline{x}_{n+1}, \overline{x}_{n+2}) \not \in \pyr(\Delta, \lambda, (q,0))$. 
So assume $(\overline{x}, \overline{x}_{n+1}, \overline{x}_{n+2})\in S\times \Z_+\times\Z_+$. Let $\overline{r} = \overline{x}-\overline{x}_{n+1}p^*+\overline{x}_{n+2}q$, $\overline{z}_1 = \overline{x}_{n+1}$, $\overline{z}_2 = \overline{x}_{n+2}$ and $\overline{s}_r = 1$ if $r=\overline{r}$ and $\overline{s}_r=0$ otherwise. Note that
\begin{equation*}
\sum_{r\in \R^n} r\overline{s}_r +p^*\overline{z}_1+q\overline{z}_2 = \overline{x}\in S.
\end{equation*}
Since~\eqref{eq:lambda_lift} is valid for \eqref{eq:int_model}, it follows that
\begin{align*}
1&\leq \sum_{r\in \R^n} \psi(r)\overline{s}_r +V_{\psi}(p^*)\overline{z}_1+\lambda\overline{z}_2\\
& = \psi(\overline{r})+V_\psi(p^*)\overline{x}_{n+1}+\lambda\overline{x}_{n+2}\\
& = \max_{i \in I} \left\{a_i\cdot(\overline{x}-\overline{x}_{n+1}p^*-\overline{x}_{n+2}q )+V_{\psi}(p^*)\overline{x}_{n+1}+\lambda\overline{x}_{n+2}\right\}\\
& = \max_{i \in I} \left\{a_i\cdot\overline{x}+(V_{\psi}(p^*)-a_i\cdot p^*)\overline{x}_{n+1}+ (\lambda-a_i\cdot q)\overline{x}_{n+2}\right\}.
\end{align*}
Hence, $\pyr(\Delta, \lambda, (q,0))$ is $(S\times\Z\times\Z)$-free.
\end{cpf}

\smallskip

\noindent The converse of the Claim~\ref{claim:lift_lambda_once} is also true.

\begin{CLAIM}
\label{claim:valid_big_set}
If $\lambda>0$ and $\pyr(\Delta, \lambda, (q,0))$ is $(S\times\Z\times\Z)$-free, then \eqref{eq:lambda_lift} is valid for \eqref{eq:int_model}.
\end{CLAIM}
\begin{cpf}
Consider $\Psi : \R^n \times \R \times \R \to \R$ defined by
\begin{equation*}
\Psi(r, r_{n+1}, r_{n+2}) := \max_{i\in I}\left\{a^i\cdot r + (V_{\psi}(p^*) - a^i\cdot p^*)r_{n+1} +(\lambda - a^i\cdot q)r_{n+2}\right\}.
\end{equation*}
Let $(s, y_{p^*}, y_q) \in \eqref{eq:int_model}$. From the observation above, $(s, y_{p^*}, y_q) \in \eqref{eq:cts_model}$. 
Note that $\Psi(r,0,0) = \psi(r), \Psi(p^*,1,0) = V_{\psi}(p^*)$, and $\Psi(q,0,1) = \lambda$. 
It follows that
\begin{align*}
& \sum_{r\in \R^n}\psi(r)s_r+V_{\psi}(p^*)y_{p^*}+\lambda y_q\\
= & \sum_{r\in \R^n}\Psi(r,0,0) s_r +\Psi(p^*,1,0) y_{p^*} +\Psi(q,0,1)y_q ~ \geq ~ 1 .
\end{align*}
Hence, \eqref{eq:lambda_lift} is valid for \eqref{eq:int_model}.
\end{cpf}

By Theorem~\ref{thm:extension} with $\cR = \R^n$ and $\cP = \{p^*_1, p^*_2\}$, $V_\psi(p^*_2;p^*_1)$ is the infimum of $\lambda > 0$ such that~\eqref{eq:lambda_lift} is valid for~\eqref{eq:int_model}. The result now follows from Claims~\ref{claim:lift_lambda_once} and~\ref{claim:valid_big_set}.
\end{proof}

\section{Proof of Proposition~\ref{prop:projection}}\label{AppendixA1}

\begin{proof}[Proof of Proposition~\ref{prop:projection}] 
From~\eqref{eq:Blambda} and~\eqref{eq:R_B-x}, we have 
\begin{equation*}
\begin{array}{rl}
&
R_{\Delta}(\bar x,\bar x_{n+1})\\\\
 =& \left\{\begin{array}{rcl}\hspace{-.15 cm}(r,r_{n+1}): & \hspace{-.375 cm}\begin{array}{l} (a^i - a^k)\cdot r + r_{n+1}((a^k - a^i)\cdot p^*) \leq 0 \text{ and}\\
(a^i - a^k)\cdot(\bar x - r) + (\bar x_{n+1} - r_{n+1})((a^k - a^i)\cdot p^*) \leq 0 ~~ \forall i \in I
\end{array} 
\end{array} \right\},
\end{array}
\end{equation*}
where $k \in I$ is such that $\psi(x) = a^k\cdot x$.
Therefore, 
$$
\begin{array}{rl}
&H_t\cap R_{\Delta}(\bar x,\bar x_{n+1})\\\\ 
 = & \left\{\begin{array}{rcl}(r,t): & \begin{array}{l} (a^i - a^k)\cdot r + t((a^k - a^i)\cdot p^*) \leq 0 \text{ and} \\
(a^i - a^k)\cdot(\bar x - r) + (\bar x_{n+1} - t)((a^k - a^i)\cdot p^*) \leq 0 ~~ \forall ~ i \in I
\end{array}  \end{array}\right\}\\\\
= & \left\{\begin{array}{rcl}(\tilde r + tp^*,t): & \begin{array}{l} (a^i - a^k)\cdot \tilde r \leq 0 \text{ and} \\
(a^i - a^k)\cdot(\bar x - \tilde r) + \bar x_{n+1}((a^k - a^i)\cdot p^*) \leq 0 ~~ \forall ~ i \in I
\end{array} \end{array}\right\}\\\\
= & (H_0 \cap R_{\Delta}(\bar x,\bar x_{n+1}))+t(p^*, 1).
\end{array}
$$
A similar calculation shows 
\(
 H_0 \cap R_{\Delta}(\bar x,\bar x_{n+1}) = R_B(\bar{x} - \bar{x}_{n+1}p^*) \times \{0\}.
\)
~
\end{proof}

\section{Proof of Theorem~\ref{thm:fixing_covering_trans}}\label{sec:translation}

The first lemma required for Theorem~\ref{thm:fixing_covering_trans} is an extension of the so-called `Collision Lemma' (Lemma 3.2 in~\cite{basu-paat-lifting}). 
In this appendix, set $\Delta := \pyr(B, V_{\psi}(p^*), p^*).$

\begin{proposition}\label{prop:new_collision} 
Let $B$ be a maximal $S$-free convex $0$-neighborhood in $\R^n$ of the form~\eqref{eqBDefn}. 
Let $p^*\in \R^n$ and $(\overline{x}, \overline{x}_{n+1}), (\overline{y}, \overline{y}_{n+1}) \in \Delta\cap ({S\times \Z})$, and $i_x, i_y\in I$ satisfy $(a^{i_x}, V_{\psi}(p^*)-a^{i_x}\cdot p^*)\cdot(\overline{x}, \overline{x}_{n+1})=(a^{i_y}, V_{\psi}(p^*)-a^{i_y}\cdot p^*)\cdot(\overline{y}, \overline{y}_{n+1})=1$. 
Let $k_x,k_y\in \Z$ be such that $0\leq k_x\leq\overline{x}_{n+1}$ and $0\leq k_y\leq \overline{y}_{n+1}$, and let $(x, k_x)\in R_{\Delta}(\overline{x}, \overline{x}_{n+1})$ and $(y, k_y)\in R_{\Delta}(\overline{y}, \overline{y}_{n+1})$. 
If $x-y\in W_S$, then $$(a^{i_x}, V_{\psi}(p^*)-a^{i_x}\cdot p^*)\cdot(x, k_x) =(a^{i_y}, V_{\psi}(p^*)-a^{i_y}\cdot p^*)\cdot(y, k_y).$$ 

If $(x, k_x)\in \intr\left(R_{\Delta}(\overline{x}, \overline{x}_{n+1})\right)$ and $(y, k_y)\in \intr\left(R_{\Delta}(\overline{y}, \overline{y}_{n+1})\right)$, then $(a^{i_x}, V_{\psi}(p^*)-a^{i_x}\cdot p^*) = (a^{i_y}, V_{\psi}(p^*)-a^{i_y}\cdot p^*)$. 

\end{proposition}

\begin{proof}
Let $(x, k_x)\in R_{\Delta}(\overline{x}, \overline{x}_{n+1})$ and $(y, k_y)\in R_{\Delta}(\overline{y}, \overline{y}_{n+1})$. 
Assume to the contrary that $(a^{i_x}, V_{\psi}(p^*)-a^{i_x}\cdot p^*)\cdot(x, k_x) < (a^{i_y}, V_{\psi}(p^*)-a^{i_y}\cdot p^*)\cdot(y, k_y)$ and consider $(\overline{y}, \overline{y}_{n+1})+(x-y, k_x-k_y)$ (if the inequality is reversed then consider $(\overline{x}, \overline{x}_{n+1})+(y-x, k_y-k_x)$ instead). 
Since $x-y\in W_S$ and $k_y\leq \overline{y}_{n+1}$, it follows that $(z, z_{n+1}):=(\overline{y}, \overline{y}_{n+1})+(x-y, k_x-k_y) = (\overline{y}+(x-y), (\overline{y}_{n+1}-k_y)+k_x)\in {S\times \Z}$. 
We claim that $(z, z_{n+1})\in \intr(\Delta)$, contradicting that $\Delta$ is {$(S\times \Z)$}-free. 
We will show this using the half-space definition of $\Delta$ from~\eqref{eq:Blambda}. 

Take $i\in I$ and define
$\alpha_i:= (a^{i_x}, V_{\psi}(p^*)-a^{i_x}\cdot p^*).$
If $i = i^x$, then
\begin{align*}
\alpha_{i^x}\cdot (z, z_{n+1})
&\leq 1- \alpha_{i^x}(y, k_y)
+\alpha_{i^x}\cdot(x, k_x) &&\text{since } (\overline{y}, \overline{y}_{n+1})\in S\times \Z_+\\
&< 1- \alpha_{i^y}(y, k_y)
+\alpha_{i^x}\cdot(x, k_x)&&\text{since } (y, k_y)\in  R_{\Delta}(\overline{y}, \overline{y}_{n+1})\\
&\leq 1 && \text{since }a_{i^x}\cdot(x, k_x) < a_{i^y}\cdot(y, k_y). 
\end{align*}
If $i = i^y$, then
\begin{align*}
\alpha_{i_y}\cdot (z, z_{n+1}) & = 
1- \alpha_{i_y}(y, k_y)
+\alpha_{i_y}\cdot(x, k_x) &&\text{since } (\overline{y}, \overline{y}_{n+1})\in S\times \Z_+\\
&< 1- \alpha_{i_x}(x, k_x)
+\alpha_{i_x}\cdot(x, k_x) && \text{since }a_{i^x}\cdot(x, k_x) < a_{i^y}\cdot(y, k_y)\\
&=1. 
\end{align*}
If $i\in I\setminus \{i^x, i^y\}$, then
\begin{align*}
\alpha_{i}\cdot (z, z_{n+1}) & \leq 1+\alpha_i\cdot(x, k_x)-\alpha_i\cdot(y, k_y) &&\text{since } (\overline{y}, \overline{y}_{n+1})\in S\times \Z_+\\
& \leq 1+\alpha_i\cdot(x, k_x)-\alpha_{i^y}\cdot(y, k_y) &&\text{since } (y, k_y)\in  R_{\Delta}(\overline{y},\overline{y}_{n+1})\\
& < 1+\alpha_i\cdot(x, k_x)-\alpha_{i^x}\cdot(x, k_x) && \text{since }a_{i^x}\cdot(x, k_x) < a_{i^y}\cdot(y, k_y)\\
& < 1&&\text{since } (x, k_x)\in  R_{\Delta}(\overline{x},\overline{x}_{n+1}).
\end{align*}
Hence, $(z, z_{n+1})\in \intr(\Delta)$ gives a contradiction.

Assume $(x, k_x)\in \intr\left(R_{\Delta}(\overline{x}, \overline{x}_{n+1})\right)$ and $(y, k_y)\in \intr\left(R_{\Delta}(\overline{y}, \overline{y}_{n+1})\right)$.
Assume to the contrary that $\alpha_{i^x}\neq \alpha_{i^y}$. 
Since $\alpha_{i^x}\neq \alpha_{i^y}$ and $(y, k_y)\in \intr\left(R_{\Delta}(\overline{y}, \overline{y}_{n+1})\right)$, we have
\(
\alpha_{i^x}\cdot (y, k_y)< \alpha_{i^y}\cdot(y, k_y)
\)
and
\begin{equation}\label{eq:coll_leq_2}
\alpha_{i_x}\cdot (\overline{y}-y, \overline{y}_{n+1}-k_y)< \alpha_{i_x}\cdot (\overline{y}-y, \overline{y}_{n+1}-k_y).
\end{equation}
From the previous argument that $\alpha_{i^x}(x, k_x) = \alpha_{i^y}(y, k_y)$. 
Let $i\in I$. 
If $i = i^x$, then
\begin{align*}
\alpha_{i^x}\cdot (z, z_{n+1}) & = 
\alpha_{i^x}\cdot (\overline{y} - y, \overline{y}_{n+1}- k_y)
+\alpha_{i^x}\cdot(x, k_x)\\
&< \alpha_{i^y}\cdot (\overline{y} - y, \overline{y}_{n+1}- k_y)
+\alpha_{i^x}\cdot(x, k_x) &&\text{from}~\eqref{eq:coll_leq_2}\\
& = 1-\alpha_{i^y}(y, k_y)+\alpha_{i^x}\cdot(x, k_x) &&\text{since } (\overline{y}, \overline{y}_{n+1})\in S\times \Z_+\\
& =1.
\end{align*}
If $i = i^y$, then
\begin{align*}
\alpha_{i^y}\cdot (z, z_{n+1})&= 1 -\alpha_{i^y}(y, k_y)
+\alpha_{i^y}\cdot(x, k_x)&&\text{since } (\overline{y}, \overline{y}_{n+1})\in S\times \Z_+\\
& = 1-\alpha_{i^x}(x, k_x)+\alpha_{i^y}\cdot(x, k_x) \\
& <1 && \text{since } (x, k_x)\in  R_{\Delta}(\overline{x},\overline{x}_{n+1}).
\end{align*}
If $i\in I\setminus\{i^x, i^y\}$, then 
\begin{align*}
&~ \alpha_{i}\cdot (z, z_{n+1})\\ =&~ 
\alpha_{i}\cdot (\overline{y}-y , \overline{y}_{n+1}-k_y)
+\alpha_{i}\cdot(x, k_x)\\
<&~ \alpha_{i^y}\cdot (\overline{y}-y , \overline{y}_{n+1}-k_y)
+\alpha_{i}\cdot(x, k_x)&&\text{since }(y, k_y)\in  R_{\Delta}(\overline{y},\overline{y}_{n+1})\\
<&~ 1-\alpha_{i^y}(y, k_y)+\alpha_{i^x}\cdot(x, k_x)&&\text{since }(x, k_x)\in  R_{\Delta}(\overline{x},\overline{x}_{n+1}) \\
=&~1.
\end{align*}
This shows $(z, z_{n+1})\in  \intr(\Delta)$, which is a contradiction.
\end{proof}

\begin{lemma}\label{lemma:patching-2}[Theorem 9.4 in~\cite{Degundji1966}]
Let $P_\omega \subseteq \R^n, \omega \in \Omega$ be a (possibly infinite) family of polyhedra such that any bounded set intersects only finitely many polyhedra, and $\bigcup_{\omega\in \Omega} P_\omega = \R^n$. Suppose there is a family of functions $A_\omega: P_\omega \to \R^n, \omega \in \Omega$ such that $A_\omega$ is continuous over $P_\omega$ for each $\omega\in \Omega$, and for every pair $\omega_1, \omega_2 \in \Omega$, $A_{\omega_1}(x) = A_{\omega_2}(x)$ for all $x \in P_{\omega_1} \cap P_{\omega_2}$. Then there is a unique, continuous map $A:\R^n \to \R^n$ that equals $A_\omega$ when restricted to $P_\omega$ for each $\omega\in \Omega$.\end{lemma}

\begin{proof} This follows from a direct application of Theorem 9.4 in Chapter III of~\cite{Degundji1966} by noting that polyhedra are closed sets.
\end{proof}

\begin{proposition}\label{prop:finite_intersection}
Let $B$ be a maximal $S$-free convex $0$-neighborhood in $\R^n$ such that $\intr(B\cap \conv(S)) \neq \emptyset$. Then any bounded set $U\subseteq \R^n$ intersects a finite number of polyhedra from $\mathcal X(B, p^*)+W_S$. 
\end{proposition}

\begin{proof}
Recall that $B$ is a full-dimensional set, so, by construction, $\Delta$ is full-dimensional.
Also, $\intr(\conv(S)\cap B)\neq \emptyset$ and $\intr(\conv({S\times \Z})\cap \Delta)\neq \emptyset$. 
Set $\tilde{U}:= U\times [0,1]\subseteq \R^{n+1}$. 
$\tilde{U} \subseteq \R^{n+1}$ is bounded, and by Theorem 2.7 in~\cite{basu-paat-lifting}, $\tilde{U}$ intersects finitely many polyhedra from $R(\Delta)+W_{S\times \Z_+} = R(\Delta)+W_S\times \{0\}$. 
Say for $i=1, \dots, k$, $\tilde{U}$ intersects $\tilde{P}_i+(w_i,0)$, $(w_i,0)\in W_S\times\{0\}$ and $\tilde{P}_i$ is a polyhedron in $R(\Delta)$. 

For $t\in \Z$, Proposition~\ref{prop:projection} states that the projection of $H_t\cap (\tilde{P}_i+(w_i,0))$ onto $\R^n$ is {$\proj_{\R^n}(H_0\cap \tilde{P}_i)+tp^*+w_i$, where $\proj_{\R^n}(\cdot)$} denotes the projection onto the first $n$ coordinates. 
By definition of $\mathcal X(B, p^*)+W_S$, all polyhedra in $\mathcal X(b, p^*)+W_S$ are of the form $\proj_{\R^n}(H_0\cap \tilde{P}_i)+tp^*+w_i$, where $t \le x_{n+1}$ for some blocking point $(x, x_{n+1}) \in \R^n\times \R$ corresponding to $\Delta$. 
Since $\tilde{U}$ is bounded, $H_t\cap\tilde{U}\cap (\tilde{P}_i+(w_i,0))\neq \emptyset$ for only a finite number of integral $t$, for each $i=1, \dots, k$. 
Hence, $U$ only intersects a finite number of polyhedra from $\mathcal X(B, p^*)$. 
\end{proof}

For $\epsilon>0$ and $x\in \R^d$, define $D(x; \epsilon) : = \{y\in\R^d:\|x-y\|<\epsilon\}$. 

\begin{proposition}\label{prop:closed}
Let $B$ be a maximal $S$-free convex $0$-neighborhood in $\R^n$. 
For $p^*\in \R^n$, the set $\mathcal X(B, p^*)+W_S$ is closed.
\end{proposition}

\begin{proof}
Let $x\not\in \mathcal X(B, p^*)+W_S$ and consider $D(x,1)$. 
From Proposition~\ref{prop:finite_intersection}, $D(x,1)$ intersects only finite many polyhedra $P_1, \dots, P_k$ from $\mathcal X(B, p^*)+W_S$. 
Each $P_i$ is closed, so the finite union $\cup_{i=1}^k P_i$ is too. 
Since $x\not\in \cup_{i=1}^k P_i$, there exists $\epsilon>0$ such that $D(x; \epsilon)\subseteq D(x;1)$ does not intersect $P_i$ for $i=1, \dots, k$. 
So, $D(x; \epsilon)\cap (\mathcal X(B, p^*)+W_S) = \emptyset$. 
This implies $\R^n\setminus (\mathcal X(B, p^*)+W_S)$ is open, so $\mathcal X(B, p^*)+W_S$ is closed. 
\end{proof}

Let $t$ be as in Theorem~\ref{thm:fixing_covering_trans}. For $i\in I$, set $a^i_t := \frac{a^i}{1+a^i\cdot t}$. Observe 
\begin{equation*}
B+t = \{r\in \R^n:~a^i_t\cdot r\leq 1 ~ \forall i\in I\},
\end{equation*}
and 
\begin{align*}
\Delta+(t,0)= \left\{(r, r_{n+1})\in \R^{n+1}:~a^i_t\cdot r +\left(\frac{V_{\psi}(p^*)-a^i\cdot p^*}{1+a^i\cdot t}\right)r_{n+1}\leq 1 ~ \forall i\in I\right\}.
\end{align*}
The apex of $\Delta+(t,0)$ is $\frac{1}{V_{\psi}(p^*)}(p^*+V_{\psi}(p^*)t,1)$. Define
\begin{equation}\label{def:p_hat}
\hat{p}:= p^*+V_{\psi}(p^*)t.
\end{equation}

For each $k\in \Z$, $k\geq 0$, and $i\in I$ define $T^k_i:\R^n\to \R^n$ to be $$T^k_i(x) := x+(a^i, V_{\psi}(p^*)-a^ip^*)\cdot(x,k)t.$$
The next result follows from a direct calculation.
\begin{proposition}\label{prop:invertible}
The function $T^k_i$ is invertible with the inverse defined by $$(T^k_i)^{-1}(x) = x-\left(a^i_t, \frac{V_{\psi}(p^*)-a^i\cdot p^*}{1+a^i\cdot t}\right)\cdot(x,k)t.$$
\end{proposition}

\begin{lemma}\label{lem:proj_defined}
Let $(\overline{x}, \overline{x}_{n+1}), (\overline{y}, \overline{y}_{n+1})\in \Delta\cap {(S\times\Z)}$ and $i_x, i_y\in I$ be such that $(a^{i_x}, V_{\psi}(p^*)-a^{i_x}\cdot p^*)\cdot(\overline{x}, \overline{x}_{n+1})  = (a^{i_y}, V_{\psi}(p^*)-a^{i_y}\cdot p^*)\cdot(\overline{y}, \overline{y}_{n+1}) = 1$. 
Assume $(z, k_x)\in R_{\Delta}(\overline{x}, \overline{x}_{n+1})+(w_x, 0)$ and $(z, k_y)\in R_{\Delta}(\overline{y}, \overline{y}_{n+1})+(w_y, 0)$, where $w_x, w_y\in W_S$, $k_i\in \Z_+$, $k_x\leq \overline{x}_{n+1}$, and $k_y\leq \overline{y}_{n+1}$. 
Then $T^{k_x}_{i_x}(z-w_x, k_x)+w_x = T^{k_y}_{i_y}(z-w_y, k_y)+w_y.$
\end{lemma} 

\begin{proof}
A direct calculation shows
\begin{align*}
 & ~T^{k_x}_{i_x}(z-w_x, k_x)+w_x &&\\
 =&~(z-w_x)+(a^{i_x}, V_{\psi}(p^*)-a^{i_x}\cdot p^*)\cdot(z,k_x)t+w_x && \text{by definition,}\\
=&~z+(a^{i_x}, V_{\psi}(p^*)-a^{i_x}\cdot p^*)\cdot(z,k_x)t&&\\
=&~z+(a^{i_y}, V_{\psi}(p^*)-a^{i_y}\cdot p^*)\cdot(z,k_y)t&&\text{by Proposition \ref{prop:new_collision},}\\
 =&~(z-w_y)+(a^{i_y}, V_{\psi}(p^*)-a^{i_y}\cdot p^*)\cdot(z,k_y)t+w_y && \text{by definition,}\\
 =&~T^{k_y}_{i_y}(z-w_y, k_y)+w_y&&\\
\end{align*}
\end{proof}

\begin{proposition}\label{prop:spindletospindle} Let $(\overline{x}, \overline{x}_{n+1})\in \Delta$. Consider $R_B(\overline{x}-\overline{x}_{n+1}p^*)+kp^*$ for $k\in \Z_+, k\leq x_{n+1}$. If $i_x\in I$ satisfies $(a^{i_x}, V_{\psi}(p^*)-a^{i_x}\cdot p^*)\cdot(\overline{x}, \overline{x}_{n+1}) = 1$, then 
\begin{equation*}
T^{k}_{i_x}\left(R_B(\overline{x}-\overline{x}_{n+1}p^*)+kp^*\right) = R_{B+t}(\overline{x}+t-\overline{x}_{n+1}\hat{p})+k\hat{p},
\end{equation*}
where $\hat{p}$ is defined in~\eqref{def:p_hat}.
\end{proposition}

\begin{proof} Let $y\in R_B(\overline{x}-\overline{x}_{n+1}p^*)+kp^*$. Note that $(y,k)\in R_{\Delta}((\overline{x}, \overline{x}_{n+1}))$ by Proposition~\ref{prop:projection}. 
Also, $T^{k}_{i_x}(y)\in R_{B+t}(\overline{x}+t-\overline{x}_{n+1}\hat{p})+k\hat{p}$ if and only if $(T^{k}_{i_x}(y), k)\in R_{\Delta+(t,0)}((\overline{x}+t, \overline{x}_{n+1}))$.
We will show this latter sufficient condition. 

We first show that for $i\in I$, if $[(a^i, V_{\psi}(p^*)-a^i\cdot p^*)-(a^{i_x}, V_{\psi}(p^*)-a^{i_x}\cdot p^*)]\cdot(y, k)\leq 0$, then $(a^i_t, \frac{V_{\psi}(p^*)-a^i\cdot p^*}{1+a^i\cdot t})\cdot(T^{k}_{i_x}(y),k)\leq a^{i_x}\cdot y+k\left(V_{\psi}(p^*) - a^{i_x}\cdot p^*\right)$ with equality for $i=i_x$. 
Observe that
\begin{align*}
& \left(a^i_t, \frac{V_{\psi}(p^*)-a^i\cdot p^*}{1+a^i\cdot t}\right)\cdot(T^{k}_{i_x}(y),k)\\ = &~\left(\frac{a^i}{1+a^i\cdot t}, \frac{V_{\psi}(p^*)-a^i\cdot p^*}{1+a^i\cdot t}\right)\cdot\left(y+(a^{i_x}\cdot y+(V_{\psi}(p^*)-a^{i_x}\cdot p^*)k)t,k\right)\\
=&~\frac{a^i\cdot y+k\left(V_{\psi}(p^*)-a^i\cdot p^*\right)+(a^{i_x}\cdot y)(a^i\cdot t)+(a^i\cdot t)k\left(V_{\psi}(p^*)-a^{i_x}\cdot p^*\right)}{1+a^i\cdot t} \\
\leq &~ \frac{a^{i_x}\cdot y+k\left(V_{\psi}(p^*)-a^{i_x}\cdot p^*\right)+(a^{i_x}\cdot y)(a^i\cdot t)+(a^i\cdot t)k\left(V_{\psi}(p^*)+a^{i_x}\cdot p^*\right)}{1+a^i\cdot t} \\
= &~ \frac{(1+a^i\cdot t)\left(a^{i_x}\cdot y+k\left(V_{\psi}(p^*) - a^{i_x}\cdot p^*\right)\right)}{1+a^i\cdot t}&\\
= &  a^{i_x}\cdot y+k\left(V_{\psi}(p^*) - a^{i_x}\cdot p^*\right), 
\end{align*}
where the inequality holds because $[(a^i, V_{\psi}(p^*)-a^i\cdot p^*)-(a^{i_x}, V_{\psi}(p^*)-a^{i_x}\cdot p^*)]\cdot(y, k)\leq 0$. 
Equality holds if $i=i_x$. 

Similarly, for $i\in I$ such that $[(a^i, V_{\psi}(p^*)-a^i\cdot p^*)-(a^{i_x}, V_{\psi}(p^*)-a^{i_x}\cdot p^*)]\cdot(y, k)\leq 0$, it follows that $(a^i_t, \frac{V_{\psi}(p^*)-a^i\cdot p^*}{1+a^i\cdot t})\cdot(\overline{x}+t-T^{k}_{i_x}(y),\overline{x}_{n+1}-k)\leq 1-(a^{i_x}\cdot y+(V_{\psi}(p^*)-a^{i_x}\cdot p^*))k$ with equality for $i=i_x$.

Since $(y,k)\in R_{\Delta}((\overline{x}, \overline{x}_{n+1}))$, it follows that $[(a^i, V_{\psi}(p^*)-a^i\cdot p^*)-(a^{i_x}, V_{\psi}(p^*)-a^{i_x}\cdot p^*)]\cdot(y, k)\leq 0$ for each $i\in I$. 
Applying the arguments to each $i\in I$, with equality for $i=i_x$, we see that 
$$
\left[(a^i_t, \frac{V_{\psi}(p^*)-a^i\cdot p^*}{1+a^i\cdot t})-(a^{i_x}_t, \frac{V_{\psi}(p^*)-a^{i_x}\cdot p^*}{1+a^{i_x}\cdot t})\right]\cdot(T^{k}_{i_x}(y),k)\leq 0,
$$ 
and 
$$
\left[(a^i_t, \frac{V_{\psi}(p^*)-a^i\cdot p^*}{1+a^i\cdot t})-(a^{i_x}_t, \frac{V_{\psi}(p^*)-a^{i_x}\cdot p^*}{1+a^{i_x}\cdot t})\right]\cdot(\overline{x}_{n+1}+t-T^{k}_{i_x}(y), \overline{x}_{n+1}-k)\leq 0.
$$
Hence, $(T^{k}_{i_x}(y),k)\in R_{\Delta+(t,0)}((\overline{x}+t, \overline{x}_{n+1}))$, so 
$$
T^{k}_{i_x}\left(R_B(\overline{x}-\overline{x}_{n+1})+kp^*\right) \subseteq R_{B+t}(\overline{x}+t-\overline{x}_{n+1}\hat{p})+k\hat{p}.
$$ 
Using similar reasoning applied to $(T^{k}_{i_x})^{-1}$, we get the reverse inclusion. 
\end{proof}

\begin{proof}[Proof of Theorem~\ref{thm:fixing_covering_trans}]
Recall $\hat{p}$ in~\eqref{def:p_hat}. We show if $\mathcal X(B,p^*)+W_S=\R^n$, then $\mathcal X(B+t,\hat{p})+W_{S+t}=\R^n$. 
The converse is proved by switching the roles of $(B, p^*)$ and $(B+t, \hat{p})$. 

A direct calculation shows that $W_S = W_{S+t}$ (see Proposition 2.1 in~\cite{basu-paat-lifting}). 
If $B$ is a half-space, then the lifting region is equal to $\R^n$. 
The extended lifting region is contained in $\mathcal X(B, p^*)+W_S$, so $\mathcal X(B, p^*)+W_S= \mathcal X(B+t,\hat{p})+W_{S+t}=\R^n$. 
Thus, assume that $B$ is not a half-space. 

Define the map $A:\R^n\to \R^n$ by
$$
A(y) := T^{k}_{i_x}(y-u)+u,~~~\text{if }y\in R_B(w(z))+kp^*+u,
$$
where $z = (\overline{x}, \overline{x}_{n+1})$ is a blocking point of $\Delta$, $k\in \{0, \dotsc, \overline{x}_{n+1}\}, u\in W_S$, and $\left(a^{i_x}_t, V_{\psi}(p^*)-a^{i_x}\cdot p^*\right)\cdot (\overline{x}, \overline{x}_{n+1})=1$. 
Since $\mathcal X(B, p^*)+W_S = \R^n$, each $y$ is in some $R_B(\overline{x}-\overline{x}_{n+1}p^*)+kp^*+u$.
 $A$ is well defined from Lemma~\ref{lem:proj_defined}.

By assumption, $\R^n = \mathcal X(B, p^*)+W_S$. 
Using Proposition~\ref{prop:spindletospindle}, we have
\begin{align*}
A(\R^n) =& ~ A\left(\mathcal X(B, p^*)+W_S \right)\\
=& ~ A\bigg(\bigcup_{(\bar x, \bar x_{n+1}) \in \Delta\cap (S\times \Z_+), u\in W_S} \bigg(\bigcup_{i=0}^{\overline{x}_{n+1}} (R_B( \bar x - \bar x_{n+1}p^*) + ip^*+u)\bigg)\bigg)\\
=& ~  \bigcup_{(\bar x, \bar x_{n+1}) \in \Delta\cap (S\times \Z_+), u\in W_S} \bigg(\bigcup_{i=0}^{\overline{x}_{n+1}} A(R_B( \bar x - \bar x_{n+1}p^*) + ip^*+u)\bigg)\\
=&  ~\bigcup_{(\bar x, \bar x_{n+1}) \in \Delta\cap (S\times \Z_+), u\in W_S} \bigg(\bigcup_{i=0}^{\overline{x}_{n+1}}R_{B+t}(\overline{x}+t-\overline{x}_{n+1}\hat{p})+i\hat{p}+u\bigg)\\
=& ~ \bigg(\bigcup_{(\bar x, \bar x_{n+1}) \in \Delta\cap (S\times \Z_+)} \bigg(\bigcup_{i=0}^{\overline{x}_{n+1}} R_{B+t}(\overline{x}+t-\overline{x}_{n+1}\hat{p})+i\hat{p}\bigg)\bigg)+W_{S+t}\\
=&~ \mathcal X(B+m,\hat{p})+W_{S+t}.
\end{align*}

So, $A$ maps the translated fixing region to the translated fixing region. 

Suppose $A(y_1)=A(y_2)$ for some $y_1,y_2\in \R^n$. 
Let $\alpha:= A(y_1)=A(y_2)$. By definition, for $j=1,2$, there exists a blocking point $(\overline{x}^j, \overline{x}^j_{n+1})\in S\times \Z_+$, $k_j\in \Z_+$ with $k_j\leq \overline{x}^j_{n+1}$, and $w_j\in W_S$ such that $y_j\in R_B(\overline{x}^j - \overline{x}^j_{n+1}p^*)+k_jp^*+w_j$. 
Moreover 
$$
\alpha = A(y_1) = T^{k_1}_{i_{x_1}}(y_1-w_1)+w_1 = T^{k_2}_{i_{x_2}}(y_2-w_2)+w_2 = A(y_2).
$$ 
By Proposition~\ref{prop:spindletospindle}, $\alpha\in R_{B+t}(\overline{x}^j+t-\overline{x}^j_{n+1}\hat{p})+k_j\hat{p}+w_j$, for $j\in \{1,2\}$.
So, $(\alpha, k_j)\in R_{\Delta+(t,0)}((\overline{x}^j+t, \overline{x}^j_{n+1}))+(w_j,0)$, for $j\in \{1,2\}$. 
Lemma~\ref{lem:proj_defined} applied to $(T^{k_1}_{i_{x_1}})^{-1}$ and $(T^{k_2}_{i_{x_2}})^{-1}$ shows
$$ (T^{k_1}_{i_{x_1}})^{-1}\left(T^{k_1}_{i_{x_1}}(y_1-w_1)+w_1-w_1\right) +w_1=  (T^{k_2}_{i_{x_2}})^{-1}\left(T^{k_2}_{i_{x_2}}(y_2-w_2)+w_2-w_2\right)+w_2.$$
Applying the definition of $\left(T^{k_j}_{i_{x_j}}\right)^{-1}$ for $j=1,2$, we see $y_1=y_2$. 
Hence, $A$ is injective. 

By Lemma~\ref{lemma:patching-2} and Proposition~\ref{prop:finite_intersection}, $A$ is continuous. 
The Invariance of Domain Theorem (see ~\cite{brouwer1911beweis,Dold1995}) states that $A$ is an open map. So, the translated fixing region is open because $A$ maps $\R^n$ to the translated fixing region.  
By Proposition~\ref{prop:closed}, the translated fixing region is also closed.
Because the translated fixing region is nonempty, this implies that it must be $\R^n$. Thus, $B+t$ is one point fixable. 
\end{proof}

\section{Case Analysis for $K_i$ from Claim~\ref{eq:H_i-cases}}
\label{appendix:case_analysis}

\begin{proof}[Proof of Claim~\ref{eq:H_i-cases}]

To prove this claim, we first construct the half-space definition of the spindles $R_{T}(s^4-p^*), R_{T}(s^5-p^*),$ and $R_{T}(s^6-2p^*)$. Consider the vectors $q^1, q^2,$ and $q^3$ that define $T$, see~\eqref{eqTriQvals}. 
Since $(s^4,z^4)=(s^4,1)\in P$ is contained in the same facet as $(s^1, 0)$ (see the discussion following~\eqref{eq:P}), we see that
\begin{equation}\label{defn:R_t_4}
R_{T}(s^4-p^*)= \{x\in \R^2: (q^i-q^1)\cdot x\leq 0, (q^i-q^1)\cdot (s^4-p^*-x)\leq 0 ~ \forall~i\in \{2,3\}\}.
\end{equation}
Similarly, because $(s^5,1)$ and $(s^6,2)$ share a facet with $(s^2,0)$ and $(s^3,0)$, respectively,

\begin{align}
R_{T}(s^5-p^*)&=  \{x\in \R^2: (q^i-q^2)\cdot x\leq 0, (q^i-q^2)\cdot (s^5-p^*-x)\leq 0 ~\forall  i\in \{1,3\}\}\label{defn:R_t_5}\\
R_{T}(s^6-2p^*) & = \{x\in \R^2: (q^i-q^3)\cdot x\leq 0, (q^i-q^3)\cdot (s^6-2p^*-x)\leq 0 ~ \forall  i\in \{1,2\}\}\nonumber.
\end{align}

Consider the collection of points $K_1=\conv\{l, e^1, g, u\}$. In order to prove $K_1\subseteq R_{T}(s^4-p^*)$, is is enough to show that $\{l, e^1, g, u\}\subseteq R_{T}(s^4-p^*)$. Consider the point $l\in \{l, e^1, g, u\}$. Using the values in Figure~\ref{fig:holes} along with~\eqref{defn:R_t_4} and the definition $s^4=(1+b_1, 2+b_2)$, it is straight forward, yet tedious, to show that the four values $(q^i-q^1)\cdot l, (q^i-q^1)\cdot (s^4-p^*-l),~i\in \{2,3\}$ are all contained in 
\begin{equation*}
Q := \left\{\begin{array}{c}
0,~~
\frac{-1}{(1, \gamma_1)\cdot(1+b_1, b+b_2)},~~
\frac{-\gamma_1}{(1, \gamma_1)\cdot(1+b_1, b+b_2)},~~
\frac{-1+\gamma_2}{(-1, \gamma_2)\cdot(b_1, b_2)},\\\\
\frac{-2+\gamma_3}{(-2\gamma_3, -2)\cdot(b_1, b_2)},~~
\frac{\gamma_3}{(-2\gamma_3, 2)\cdot(b_1, b_2)},~~
\frac{-1+\gamma_3}{(\gamma_3, -1)\cdot(b_1, b_2)},\\\\
\frac{-1-\gamma_1}{(1, \gamma_1)\cdot(1+b_1, 1+b_2)},~~
\frac{-b_1(1+\gamma_1\gamma_3)-(1+\gamma_1)}{(1+b_1+\gamma_1(1+b_2))(-b_2+b_1\gamma_3)},~~
\frac{b_1(-1+\gamma_3)+b_2(-1+\gamma_2)+\gamma_2}{(b_1-(1+b_2)\gamma_2)(-b_2+b_1\gamma_3)},\\\\
\frac{-(b_1+1)+\gamma_1b_2-(\gamma_1+2\gamma_1^2\gamma_3)}{(1+b_1+\gamma_1(1+b_2))(-b_2+b_1\gamma_3)},~~
\frac{\gamma_2+b_1(-1+\gamma_2\gamma_3)}{(b_1-(1+b_2)\gamma_2)(-b_2+b_1\gamma_3)}
\end{array}\right\}.
\end{equation*}

Because $\gamma_1, \gamma_2, \gamma_3>0$, $\gamma_2, \gamma_3<1$, and $-1\leq b_2\leq b_1\leq 0$, a direct calculation shows that every value in $Q$ is nonpositive. Hence, from~\eqref{defn:R_t_4}, $l\in R_{T}(s^4-p^*)$. Similar arguments show that when the four inner products defining~\eqref{defn:R_t_4} are evaluated at any point in $ \{l, e^1, g, u\}$, the result is in $Q$. Hence $ \{l, e^1, g, u\}\subseteq  R_{T}(s^4-p^*)$.

The inclusions $K_2\subseteq R_T(s^5-p^*)+(1,1)$, $K_3 \subseteq R_{T}(s^4-p^*)+p^*$, $K_4 \subseteq R_{T}(s^5-p^*)+p^*$, and $K_5\subseteq R_{T}(s^6-2p^*)+p^*$ use similar proofs. So, we only prove $K_2\subseteq R_T(s^5-p^*)+(1,1)$.  For this, it is enough to show that $\{u-(1,1), m-(1,1), i-(1,1), g-(1,1)\}\subseteq R_T(s^5-p^*)$. However, substituting these four values in for $x$ in~\eqref{defn:R_t_5} yields values in $Q$. Hence, $\{u-(1,1), m-(1,1), i-(1,1), g-(1,1)\}\subseteq R_{T}(s^5-p^*)$. 
\end{proof}

\end{document}